\documentclass[11pt,reqno]{amsproc}
\usepackage{amsmath, amsthm, amscd, amsfonts, amssymb, graphicx, color}
\usepackage{amssymb}
\usepackage[bookmarksnumbered, plainpages]{hyperref}
\newtheorem{theorem}{Theorem}[section]
\newtheorem{thm}[theorem]{\bf{Theorem}}
\newtheorem{cor}[theorem]{\bf{Corollary}}

\newtheorem{prop}[theorem]{\bf{Proposition}}
\newtheorem{defn}[theorem]{\bf{Definition}}
\newtheorem{remark}[theorem]{\bf{Remark}}
\newtheorem{ex}[theorem]{\bf{Example}}

\numberwithin{equation}{section}

\begin{document}
%------------------------------------------------------------------------------------%

%------------------------------------------------------------------------------------%
\title[\tiny{A Weak Form of Amenability of Topological Semigroups and its Applications}]{A Weak Form of Amenability of Topological Semigroups and its Applications in Ergodic and Fixed Point Theories}
\author{Ali Jabbari$^1$}
\address{Young Researchers and Elite Club, Islamic Azad University, Ardabil Branch}
\email{jabbari\underline{ }al@yahoo.com}
\author{Ali Ebadian$^2$}
\address{
 $^{2}$Department of Mathematics, Urmia University  \newline
\indent  Tehran, Iran}

\email{ebadian.ali@gmail.com}
\author{Madjid Eshaghi Gordji$^3$}
\address{$^3$Department of Mathematics, Semnan University\newline
\indent P. O. Box 35195-363, Semnan, Iran}
\email{madjid.eshaghi@gmail.com}

%-------------------------------------------------------------------------------------------------------------------------------------------------------
%-------------------------------------------------------------------------------------------------------------------------------------------------------

%-------------------------------------------------------------------------------------------------------------------------------------------------------
%-------------------------------------------------------------------------------------------------------------------------------------------------------

\maketitle
\begin{abstract}
In this paper, we introduce a weak form of amenability on  topological semigroups that we call $\varphi$-amenability, where $\varphi$ is a character on a topological semigroup. Some basic properties of this new notion are obtained and by giving some examples, we show that this definition is weaker than the amenability of  semigroups. As a noticeable result,  for a topological semigroup $S$,  it  is shown that if $S$ is $\varphi$-amenable, then  $S$ is amenable.   Moreover,  $\varphi$-ergodicity for a topological semigroup $S$ is introduced and it is proved that under some conditions on  $S$ and a   Banach space $X$,  $\varphi$-amenability and  $\varphi$-ergodicity of any antirepresntation defined by a right action $S$ on $X$, are equivalent.   A relation between  $\varphi$-amenability of topological semigroups and existance of a common fixed point is investigated and by this relation,   Hahn-Banach property of topological semigroups in the sense of $\varphi$-amenability defined and studied.
\end{abstract}

\tableofcontents

%------------------------------------------------------------------------------------%

\pagestyle{myheadings}

%------------------------------------------------------------------------------------%

\section{ Introduction}

%------------------------------------------------------------------------------------%
Let $S$ be a semigroup. A character (semi-character) on $S$ is a nonzero map $\varphi:S\longrightarrow\mathbb{T}$ ($S\longrightarrow\mathbb{D}$) such that
 $\varphi(st)=\varphi(s)\varphi(t)$, for all $s,t\in S$. We denote the space of characters (semi-characters) on $S$ by $\Delta_S(S)$ ($\Phi_S$) and $\Delta_S(S)\subset\Phi_S$.

For every $\varphi\in\Phi_S$, the map $\widehat{\varphi}:\ell^1(S)\longrightarrow\mathbb{C}$ defined by $\widehat{\varphi}(f)=\sum_{s\in S}\varphi(s)f(s)$ ($s\in S$,  $f\in\ell^1(S)$),  is the character on $\ell^1(S)$, and indeed all of characters on $\ell^1(S)$ constructed by this method, see \cite{dls, rot} for more details.

Let $_\varphi\mathcal{M}^S$ be the class of Banach $S$-bimodules such as $X$ for which the left module action of $S$ on $X$ is given by $s\cdot x=\varphi(s)x$, for all $s\in S$ and $x\in X$. Similarly,   $\mathcal{M}_\varphi^S$ is the class of Banach $S$-bimodules $X$ for which the right module action of $S$ on $X$ is given by $x\cdot s=\varphi(s)x$, for all $s\in S$ and $x\in X$.

Let $C(S)$ be the Banach algebra of complex valued continuous bounded functions on $S$. We can consider  $C(S)$ as a member of  $_\varphi\mathcal{M}^S$ and  $\mathcal{M}_\varphi^S$, because  $\varphi$ is continuous. A function $f\in C(S)$ is called  left uniformly continuous if $\lim_\alpha \|s_\alpha\cdot f-s\cdot f\|_\infty=0$, whenever $s_\alpha\longrightarrow s$. We denote the Banach algebra of all right (left) uniformly continuous functions on $S$ by $RUC(S)$ ($LUC(S)$). Let $E$ be a linear subspace of $C(S)$ which contains the constant function $1_S$. A mean on $E$ is a functional $m\in E^*$ such that $m(1_S)=\|m\|=1$. If $E$ is closed under module actions, then the mean $m$ is called left (right) invariant if $s\cdot m=m$ ($m\cdot s=m$), for all $s\in S$.

Let $S$ be a topological semigroup, for $s\in S$,  the left translation $l_s$ of $C(S)$ by $s$ is defined by $l_sf(s')=f(ss')$, for all $f\in C(S)$ and $s'\in S$ and the right one denoted by $r_s$ such that $r_sf(s')=f(s's)$.

A semigroup $S$ is called left (right) amenable if there is a left (right) invariant mean on $RUC(S)$ ($LUC(S)$), i.e., there is a linear functional $m$ in $RUC(S)^*$ ($LUC(S)^*$) such that $m(l_s f)=m(f)$ ($m(r_sf)=m(f)$, for all $f\in RUC(S)$ ($LUC(S)$) and $s\in S$. Moreover, $S$ is called amenable if it is both left and right amenable.

Let $S$ be a topological semigroup and let $X$ be a Banach $S$-bimodule. A bounded derivation is a weak$^*$-continuous map $D:S\longrightarrow X^*$, such that $D(st)=s\cdot D(t)+D(s)\cdot t$, for all $s,t\in S$, and $\sup_{s\in S}\|D(s)\|<\infty$. The bounded derivation $D$ is called principle, if there is an element $f\in X^*$ such that $D(s)=s\cdot f-f\cdot s=\emph{\emph{ad}}_f(s)$, for all $s\in S$. If every derivation on semigroup $S$ is principle, then $S$ is called Johnson amenable. The Johnson amenability of semigroups and groups studied in \cite{pm}.

The amenability of (topological) semigroups and topological groups have  close connections with amenability of Banach algebras defined on semigroups and groups. A well-known result related to these connections, is the Johnson Theorem \cite{jo1}: the locally compact group $G$ is amenable if and only if $L^1(G)$ is amenable. A Banach algebra $\mathfrak{A}$ is said to be amenable  if, for any  Banach $\mathfrak{A}$-bimodule $X$, every continuous derivation $D:\mathfrak{A}\longrightarrow X^*$ is inner.

Let $\mathfrak{A}$ be  a Banach algebra and $\sigma(\mathfrak{A})$ is the carrier space of $\mathfrak{A}$, and  $\varphi\in\sigma(\mathfrak{A})$ is a  homomorphism from $\mathfrak{A}$ onto $\mathbb{C}$.   Assume that $\varphi\in\sigma(\mathfrak{A})\cup\{0\}$ and $X$ is an arbitrary Banach space, then $X$ can be viewed as Banach left or right $\mathfrak{A}$-module by the following actions
\begin{equation}\label{}
    \nonumber
    a\cdot x=\varphi(a)x\hspace{1cm}\emph{\emph{and}}\hspace{1cm}x\cdot a=\varphi(a)x\hspace{1cm}(a\in\mathfrak{A}, x\in X).
\end{equation}

The Banach algebra $\mathfrak{A}$ is said to be left character amenable (LCA), if
for all $\varphi\in\sigma(\mathfrak{A})\cup\{0\}$ and  Banach $\mathfrak{A}$-bimodules such as $X$ for which the left module action is
given by $a\cdot x=\varphi(a)x$ ($a\in\mathfrak{A}, x\in X$), every continuous derivation $D:\mathfrak{A}\longrightarrow X^*$ is inner.
Right character amenability (RCA) is defined similarly by considering Banach $\mathfrak{A}$-bimodules such as $X$ for which the left module action is given by $x\cdot a=\varphi(a)x$, and $\mathfrak{A}$ is called character amenable (CA) if it is both left and right character amenable. The notion of character amenability of Banach algebras was defined by
Sangani Monfared in \cite{mo1} and the concept of $\varphi$-amenability of Banach algebras introduced by Kaniuth and et al. in \cite{ka}.

Let $S$ be a topological semigroup and $0\neq\varphi\in\Delta_S(S)$. This paper, considers the concept of left (right) $\varphi$-amenability
 of  topological semigroup $S$ and some notions that have connections with it.

In section 2, we introduce left (right) $\varphi$-amenability of  topological semigroups and show that this new notion is different from  amenability of semigroups (groups). Moreover, some results about relations between $\varphi$-amenability of semigroups (groups) and character amenability of semigroup algebras (group   algebras) are obtained.

In section 3, we considers some hereditary properties of $\varphi$-amenability and define the strongly left (right)  $\varphi$-amenability on subsemigroups of semigroups where with this definition, we show that strongly left (right)  $\varphi$-amenability of a  left (right) thick subsemigroup implies that the left (right)  $\varphi$-amenability of a semigroup and vice versa.

Section 4 deals with left $\varphi$-egodicity that we introduce it in that section and investigate some relations between  left  $\varphi$-amenability of  $S$ and left $\varphi$-egodicity. Moreover, we obtain a  characterization of
 $\varphi$-amenability of  $S$ in terms of antirepresentations of $S$ on a Banach space.

 Finally, in section 5, we study of the existence of a common
 Fixed point in compact convex sets that $S$ has continuous affine actions on. Moreover, we define Hahn-Banach Property related to $0\neq\varphi\in\Delta_S(S)$ and as an interesting result, we characterize $\varphi$-amenability of  topological semigroups.

%------------------------------------------------------------------------------------------------------------------------------------------------------%
%------------------------------------------------------------------------------------------------------------------------------------------------------%
%------------------------------------------------------------------------------------------------------------------------------------------------------%
\section{$\varphi$-Amenability }
%------------------------------------------------------------------------------------------------------------------------------------------------------%
%------------------------------------------------------------------------------------------------------------------------------------------------------%
%------------------------------------------------------------------------------------------------------------------------------------------------------%
Let $S$  be a topological semigroup and $\varphi\in\Delta_S(S)$. In this section, we study  left (right) $\varphi$-amenability of the semigroup $S$ and obtain  some necessary and sufficient related to left (right) $\varphi$-amenability of $S$ such as Theorem \ref{fs1}.  We start this section with the new definition as follows
\begin{defn}
Let $S$ be a topological semigroup and $0\neq\varphi\in\Delta_S(S)$.  We say that
\begin{itemize}
  \item[(i)] $S$ is left  $\varphi$-amenable if, for each Banach $S$-bimodule $X\in _\varphi\mathcal{M}^S$, every bounded derivation $D:S\longrightarrow X^*$ is principle.
  \item[(ii)]  $S$ is right  $\varphi$-amenable if, for each Banach $S$-bimodule $X\in\mathcal{M}_\varphi^S$, every bounded derivation $D:S\longrightarrow X^*$ is principle.
  \item[(iii)]   $S$ is  $\varphi$-amenable  if it is both left and right  $\varphi$-amenable.
\end{itemize}
\end{defn}
In this paper, we suppose that $0\neq\varphi\in\Delta_S(S)$. The following result is one of the main results of this paper, indeed, this paper results are depend on.

\begin{thm}\label{t1}
Let $S$ be a topological semigroup and $\varphi\in\Delta_S(S)$. Then the following statements are equivalent
\begin{itemize}
  \item[(i)] $S$ is left \emph{(}right\emph{)} $\varphi$-amenable;
  \item[(ii)] there is a bounded linear functional $m$ in $RUC(S)^*$ \emph{(}$LUC(S)^*$\emph{)}, such that $m(\varphi)=1$ and  $m(f\cdot s)=\varphi(s)m(f)$ $(m(s\cdot f)=\varphi(s)m(f))$, for all $s\in S$ and $f\in RUC(S)$ \emph{(}$f\in LUC(S)$\emph{)}.
\end{itemize}
\end{thm}
\begin{proof}
$(\emph{\emph{i}})\Longrightarrow(\emph{\emph{ii}})$ Let $S$ be left $\varphi$-amenable. Continuity of $\varphi$ implies that the Banach algebra $RUC(S)$ is a Banach $S$-bimodule via
\begin{equation*}
    (s\cdot f)(t)=\varphi(s)f(t)\hspace{0.5cm}\emph{\emph{and}}\hspace{0.5cm}(f\cdot s)(t)=f(st),
\end{equation*}
for all $s,t\in S$ and $f\in RUC(S)$. Clearly, every $\varphi$ belongs to $RUC(S)$. It follows that $\varphi\cdot s=\varphi(s)\varphi$, for all $s\in S$. Moreover,
\begin{eqnarray*}
  (\varphi\cdot s)(t) &=& \varphi(s)\varphi(t)=\varphi(t)\varphi(s)= \varphi(ts)\\
   &=&  \varphi(s)\varphi(t),
\end{eqnarray*}
for all $s,t\in S$. Therefore, for every $s\in S$,
\begin{equation}\label{e1}
    s\cdot\varphi=\varphi\cdot s=\varphi(s)\varphi.
\end{equation}

Thus $\mathbb{C}\varphi$ is a closed $S$-subbimodule of $RUC(S)$. Consider the quotient Banach $S$-bimodule $X=RUC(S)/\mathbb{C}\varphi$.  Then put $Y=X^*\cong\{f\in RUC(S)^*:f(\varphi)=0\}\subseteq RUC(S)^*$. Let $\Phi_0\in RUC(S)^*\setminus Y$  such that $\Phi_0(\varphi)=1$ and let $\delta_{\Phi_0}:S\longrightarrow Y$ be as follows
\begin{equation*}
     \delta_{\Phi_0}(s) =s\cdot\Phi_0-\Phi_0\cdot s=s\cdot\Phi_0-\varphi(s)\Phi_0.
\end{equation*}

 Clearly, $ \delta_{\Phi_0}$ is a bounded derivation.  On the other hand,  $S$ is left $\varphi$-amenable, then  there is an element $\mathfrak{x}$ in $Y$ such that
\begin{equation}\label{e2}
    s\cdot\Phi_0-\varphi(s)\Phi_0=\delta_{\Phi_0}(s)=s\cdot\mathfrak{x}-\varphi(s)\mathfrak{x},
\end{equation}
for all $s\in S$. Let $m={\Phi_0}-\mathfrak{x}$. Clearly $m\in RUC(S)^*$, $m(\varphi)=1$ and \eqref{e2} implies $m(f\cdot s)=\varphi(s)m(f)$, for all $s\in S$ and $f\in RUC(S)$. Similarly, we can prove  right $\varphi$-amenability.

 {(ii)}$\Longrightarrow${(i)} Let  $X\in~_\varphi\mathcal{M}^S$ and  $D:S\longrightarrow X^*$ be a bounded derivation. For any $x\in X$, we  define $\omega_x:S\longrightarrow \mathbb{C}$ by $\omega_x(s)=(D(s))(x)$. Since $D$ is bounded,  $\omega_x$ is bounded and continuous. Let $t_\alpha\longrightarrow t$ in $S$. Then
\begin{eqnarray}\label{e3}
\nonumber
  \|\omega_x\cdot t_\alpha-\omega_x\cdot t\|_\infty &=& \sup_{s\in S}|\omega_x(t_\alpha s)-\omega_x(ts)|=\sup_{s\in S}|(D(t_\alpha s))(x)-(D(ts))(x)| \\
  \nonumber
   &\leq&  \sup_{s\in S}|\varphi(s)(D(t_\alpha)-D(t))(x)|+\sup_{s\in S}|(D(s))(x\cdot t_\alpha -x\cdot t)|\\
   &&\longrightarrow0.
\end{eqnarray}

Therefore, (\ref{e3}) follows that $\omega_x\in RUC(S)$. Let $m$ be a linear functional in $RUC(S)^*$ which satisfies  (ii). Define a linear functional $f\in X^*$ by $m(\omega_x)=f(x)$, for all $x\in X$. Since $D$ is a derivation,  we have
\begin{eqnarray}\label{e4}
\nonumber
   \omega_{x\cdot s}(t)&=& (D(t))(x\cdot s)=(s\cdot D(t))(x)=(D(st))(x)-(\varphi(t)D(s))(x) \\
   &=& \omega_x(st)- (D(s))(x)\varphi(t)=(\omega_x\cdot s-(D(s))(x)\varphi)(t),
\end{eqnarray}
for all $s,t\in S$ and $x\in X$. Thus (\ref{e4}) implies that $\omega_{x\cdot s}=\omega_x\cdot s-(D(s))(x)\varphi$. Then
\begin{eqnarray}\label{e5}
\nonumber
   (\varphi(s)f-s\cdot f)(x)&=& \varphi(s)f(x) -f(x\cdot s)=\varphi(s)m(\omega_x)-m(\omega_{x\cdot s})\\
   \nonumber
   &=& \varphi(s)m(\omega_x)-m(\omega_x\cdot s-(D(s))(x)\varphi) \\
   \nonumber
   &=&  \varphi(s)m(\omega_x)-m(\omega_x\cdot s)+(D(s))(x)m(\varphi)\\
   &=&(D(s))(x),
\end{eqnarray}
for all $s\in S$ and $x\in X$. Hence, $D(s)=s\cdot(-f)-\varphi(s)(-f)$, for all $s\in S$. This means that $D$ is principle and $S$ is left $\varphi$-amenable.
\end{proof}

The following example shows that a semigroup $S$ may be $\varphi$-amenable ($\varphi\in\Delta_S(S)$) but it is not amenable, and thereby, we show that the space of $\varphi$-amenable semigroups is wider than  amenable semigroups.
\begin{ex}
\begin{itemize}
  \item[(i)] Let $S$ be  a left or a right cancellative semigroup with  the identity element $e$ such that, for each $s\neq e$, $st=\varphi(t)s$, where $\varphi\in\Delta_S(S)$ and let $\dim\ell^1(S)\geq2$. Let $\{\mathcal{V}_\alpha\}$ be a collection of neighborhood basis for $e$. We construct a net $(v_\alpha)_\alpha$ from $\{\mathcal{V}_\alpha\}$ such that $v_\alpha\in\mathcal{V}_\alpha$ and $\lim_\alpha\varphi(v_\alpha)=1$. By passing
into a suitable bounded subnet $(\eta_\alpha)_\alpha\subseteq(v_\alpha)_\alpha$, we can find an element $\mu\in\beta S$ $(Stone-\check{C}ech$ compactification of $S)$ such that $\mu=\lim_\alpha\eta_\alpha$. The space $\beta S$ is homeomorphic to the character space of $\ell^\infty(S)$. Without loss of generality, we can suppose that $\mu$ belongs to character space of $\ell^\infty(S)$ and hence it belongs to $RUC(S)^*$. It is easy to check that $\mu$ satisfies  condition  \emph{(ii)} of Theorem \ref{t1}. Thus $S$ is left $\varphi$-amenable, but $S$ is not right amenable because $\ell^1(S)$ is not  amenable   \cite[Theorem 2.3]{gr}.
  \item[(ii)] Let $S$ and $T$ be semigroups. Suppose that $S$ acts on $T$ on the left; i.e., assume that there is a semigroup homomorphism $\tau$ from $T$ to $\mathrm{End}(S)$, the set of endomorphisms on $S$,  such that, for each $t\in T$ there exists $\tau_t : S\longrightarrow S$ such that $\tau_{t_1}(\tau_{t_2}(s))=\tau_{t_1t_2}(s)$, for all $t_1,t_2\in T$. Then $S\rtimes_\tau T$ is called the semidirect product of $S$ and $T$ with respect to $\tau$.  If $S\rtimes_\tau T$ is the semigroup consisting of elements of the form $(s,t)$, where $s\in S$ and $t\in T$ equipped with multiplication given by
      $$(s_1,t_1)(s_2,t_2)=(s_1\tau_{t_1}(s_2),t_1t_2),$$
      for all $(s_1,t_1),(s_2,t_2)\in S\rtimes_\tau T$. The amenability of semidirect product of two semigroups is investigated by Klawe \cite{kl}.

      Let  $S$ be a left amenable unital semigroup with the identity element $e$ such that consists of at least two elements, $m_S$ be a left invariant mean for $S$ and  $T$ be a semigroup.  Define $\tau$ from $T$ into $\mathrm{End}(S)$ by $\tau_t(s)=e$, for all $s\in S$ and $t\in T$. Thus, $(s_1,t_1)(s_2,t_2)=(s_1,t_1t_2),$
      for all $(s_1,t_1),(s_2,t_2)\in S\rtimes_\tau T$. Let $\varphi\in\Delta_T(T)$, then $\widetilde{\varphi}(s,t)=\varphi(t)$ is a character on $S\rtimes_\tau T$.  Assume that $T$ is $\varphi$-amenable.  According to \cite[Remark 3.6]{kl}, $S\rtimes_\tau T$ is not left amenable (even $T$ is left amenable). For every $f\in RUC(S\rtimes_\tau T)$, define $f_t(s)=f(s,t)$, for all $s\in S$ and $t\in T$. Clearly, $f_t\in RUC(S)$ and define $g(t)=m_S(f_t)$,  for every  $t\in T$. Thus, according to definition, we have $g\in RUC(T)$. Since $T$ is left $\varphi$-amenable, Theorem \ref{t1} implies that there exists $m_\varphi$ on $RUC(T)$ that satisfies in the stated conditions. Now define $\mathbf{m}(f)=m_\varphi(g)$, for every $f\in RUC(S\rtimes_\tau T)$. Then
      \begin{eqnarray*}
       \mathbf{m}(f\cdot(s,t))  &=& m_\varphi[m_S((f \cdot(s,t))(x,y))]=m_\varphi[m_S(f(s,ty))]\\
         &=&  m_\varphi[m_S(f_{ty}(s))]= m_\varphi[g(ty)]= m_\varphi[(g\cdot t)(y)]\\
         &=&\varphi(t)m_\varphi(g)=\varphi(t)\mathbf{m}(f)\\
         &=&\widetilde{\varphi}(s,t)\mathbf{m}(f),
      \end{eqnarray*}
    for all $f\in RUC(S\rtimes_\tau T)$ and   $(s,t)\in S\rtimes_\tau T$. The above obtained equalities imply that $\mathbf{m}(\widetilde{\varphi})=1$. This means that $S\rtimes_\tau T$ is left $\widetilde{\varphi}$-amenable.
\end{itemize}
\end{ex}

For a semigroup $S$, by $1_S$ we mean the constant function that $1_S(s)=1$, for every $s\in S$. Clearly,   $1_S\in\Delta_S(S)$ and we have the following result that gives a relationship between left (right) $1_S$-amenability and left (right)  amenability of $S$.
\begin{thm}\label{tt2}
Let $S$ be a topological semigroup, then $S$ is right \emph{(}left\emph{)} $1_S$-amenable if and only if $S$ is  left \emph{(}right\emph{)} amenable.
\end{thm}
\begin{proof}
Let $S$ be right  $1_S$-amenable and let $RUC(S)\in~_{1_S}\mathcal{M}^S$. Let $X=RUC(S)/\mathbb{C}1_S$, then $X\in~_{1_S}\mathcal{M}^S$ and $X^*$ is canonically isometrically isomorphic with the
submodule $(\mathbb{C}1_S)^\bot=\{f\in RUC(S)^*: f(1_S)=0\}$. For $f\in RUC(S)^*\setminus (\mathbb{C}1_S)^\bot$, define $D:S\longrightarrow RUC(S)^*$ by $D(s)=s\cdot f-f\cdot s$. It is  easy to check that $D$ is a derivation. Since $RUC(S)\in~_{1_S}\mathcal{M}^S$,  $D(s)=s\cdot f-1_S(s)f=s\cdot f-f$, for every $s\in S$. Right $1_S$-amenability of $S$ implies that there is an element $g\in(\mathbb{C}1_S)^\bot$ such that $D(s)=s\cdot g-g$, for all $s\in S$. Now, define $h=g-f$. Obviously, $h\neq0$ and $s\cdot h=h$. This means that $h$ is left $S$-invariant on $RUC(S)$.
The Banach algebra $RUC(S)$ is a $C^*$-subalgebra of $\ell^\infty(S)$, and Gelfand's Theorem implies that there is a compact Hausdorff space $\Omega$ such that $RUC(S)$ is isometrically $*$-isomorphic to $C(\Omega)$ as $C^*$-algebras and $S$-bimodules. Hence, we can suppose that $h$ as a $S$-invariant complex Borel regular measure on $\Omega$. Let $|h|$ be the total variation measure of $h$. Now, define $m=|h|/|h|(\Omega)$, which $m$ is a left $S$-invariant mean on $RUC(S)$. In the other words, $S$ is left amenable.

Conversely, suppose that $S$ is left amenable. Let $m$ be a left invariant mean for $S$. Let $\omega_x$ and $f\in X^*$ be as in the proof of Theorem \ref{t1}. By (\ref{e4}) we have
\begin{equation}\label{ett1}
    \omega_{x\cdot s}(t)=(\omega_x\cdot s-(D(s))(x))\varphi_S(t)=(\omega_x\cdot s-(D(s))(x)1_S)(t)
\end{equation}
for all $s,t\in S$ and $x\in X$. Thus (\ref{ett1}) implies that $\omega_{x\cdot s}=\omega_x\cdot s-(D(s))(x)1_S$. Then by a similar argument in (\ref{e5}), we conclude that $D$ is principle.
\end{proof}
\begin{defn}
Let $S$ be a topological semigroup and $\Delta_S(S)$ be the character space of $S$. We say that
\begin{itemize}
  \item[(i)]  $S$ is left  character amenable if, for each $\varphi\in\Delta_S(S)$ and  Banach $S$-bimodule $X\in\mathcal{M}_\varphi^S$, every bounded derivation $D:S\longrightarrow X^*$ is principle.
  \item[(ii)]  $S$ is right  character amenable if, for each $\varphi\in\Delta_S(S)$ and  Banach $S$-bimodule $X\in~_\varphi\mathcal{M}^S$, every bounded derivation $D:S\longrightarrow X^*$ is principle.
  \item[(iii)]  $S$  is character amenable  if it is both left and right  character amenable.
\end{itemize}
\end{defn}
In the above statements, if $\varphi\in\Phi_S$, then we say that $S$ is left (right) semi-character amenable.
Now, we consider character amenability of topological semigroups and we obtain the following result by Theorem \ref{tt2}.
\begin{cor}\label{t2}
Let $S$ be a topological semigroup. If $S$ is right \emph{(}left\emph{)} character amenable, then $S$ is  left \emph{(}right\emph{)} amenable.
\end{cor}

\begin{cor}\label{c1}
Let $S$ be a topological semigroup. If $S$ is character amenable then $S$ is amenable.
\end{cor}
\begin{cor}
Let $S$ be a unital and left or right cancellative semigroup. If $S$ is character amenable then $\ell^1(S)$ is amenable.
\end{cor}
\begin{proof}
Corollary \ref{c1} implies that $S$ is amenable and   \cite[Theorem 2.3]{gr} completes the proof.
\end{proof}

Now, this question arises that:\emph{ when character amenability of a semigroup $S$ and character amenability of $\ell^1(S)$ are equivalent}? At this time we do not know in general case, but we have the following result for discrete semigroups.
\begin{thm}
Let $S$ be a discrete semigroup. Then $S$ is $\varphi$-amenable if and only if $\ell^1(S)$ is $\widehat{\varphi}$-amenable.
\end{thm}
\begin{proof}
Let  $\varphi\in\Delta_S(S)$ and  $S$ be right $\varphi$-amenable. Then $\widehat{\varphi}:\ell^1(S)\longrightarrow\mathbb{C}$ defined by $\widehat{\varphi}(f)=\sum_{s\in S}\varphi(s)f(s)$ is a character on $\ell^1(S)$. Let $X\in~_{\widehat{\varphi}}\mathcal{M}^{\ell^1(S)}$, and let $D:\ell^1(S)\longrightarrow X^*$ be a bounded derivation. By the following actions we can see $X\in~_\varphi\mathcal{M}^S$:
\begin{equation*}
    s\cdot x=\delta_s\cdot x=\widehat{\varphi}(\delta_s) x=\sum_{t\in S}\varphi(t)\delta_s(t) x=\varphi(s) x,~\emph{\emph{and }}~ x\cdot s=x\cdot\delta_s,
\end{equation*}
for all $s\in S$ and $x\in X$. Consider the mapping $d:S\longrightarrow X^*$ by $d(s)=D(\delta_s)$. Clearly $d$ is a bounded derivation, and since $S$ is right $\varphi$-amenable,  there exists $\mathfrak{x}\in X^*$ such that $d(s)=s\cdot\mathfrak{x}-\varphi(s)\mathfrak{x}$. This implies that $D(f)=f\cdot\mathfrak{x}-\widehat{\varphi}(f)\mathfrak{x}$, for all $f\in\ell^1(S)$. Hence $\ell^1(S)$ is right $\widehat{\varphi}$-amenable. We obtain $\ell^1(S)$ is left $\widehat{\varphi}$-amenable in a similar way..

Conversely, let $\ell^1(S)$ be right $\widehat{\varphi}$-amenable. Suppose that $X\in~_\varphi\mathcal{M}^S$, and $d:S\longrightarrow X^*$ is a bounded derivation. We can consider $X$ as a Banach $\ell^1(S)$-bimodule via
\begin{equation*}
    f\cdot x=\sum_{s\in S}f(s)(s\cdot x)=\sum_{s\in S}\varphi(s)f(s)x=\widehat{\varphi}(f)x,~\emph{\emph{and}}~x\cdot f=\sum_{s\in S}f(s)(x\cdot s),
\end{equation*}
for all $f\in \ell^1(S)$ and $x\in X$. The derivation $d$ can be extended to a bounded derivation $D:\ell^1(S)\longrightarrow X^*$ with $d(s)=D(\delta_S)$ (for more details see \cite{da}, pp. 737). Since $\ell^1(S)$ is right $\widehat{\varphi}$-amenable,  there exists $\mathfrak{x}\in X^*$ such that $D(f)=f\cdot \mathfrak{x}-\widehat{\varphi}(f)\mathfrak{x}$, for all $f\in\ell^1(S)$. Then $d(s)=s\cdot \mathfrak{x}-\varphi(s)\mathfrak{x}$, for all $s\in S$. The proof for the left case is similar.
\end{proof}

In the above Theorem, if $\varphi\in\Phi_S$ and we replace semi-character amenability instead of character amenability,  we can prove  the following result.
\begin{cor}
Let $S$ be a discrete semigroup. Then $S$ is semi-character amenable if and only if $\ell^1(S)$ is character amenable.
\end{cor}
By the following Theorem, we characterize Johnson's Theorem as follows:
\begin{thm}\label{t3}
Let $G$ be a locally compact topological group. Then the following statements are equivalent
\begin{itemize}
  \item[(i)] $G$ is amenable;
  \item[(ii)] $G$ is character amenable;
  \item[(iii)] $L^1(G)$ is amenable;
  \item[(iv)] $L^1(G)$ is character amenable.
\end{itemize}
\end{thm}
\begin{proof}
(i)$\Longrightarrow$(ii) follows from  \cite[Theorem 3.7]{pm}. Corollary \ref{c1}, implies (ii)$\Longrightarrow$(i). (i)$\Longleftrightarrow$(iii) is Johnson's Theorem, and by   \cite[Corollary 2.4]{mo1}, we have (iii)$\Longleftrightarrow$(iv).
\end{proof}

Let $S$ be a topological semigroup and $f\in\ell^1(S)$ is said to be a finite mean, if $f(s)\geq0$, for every $s\in S$, $\{s:f(s)>0\}$ is finite and $\|f\|=\sum_{s\in S}f(s)=1$. Day proved that a semigroup $S$ is left amenable if and only if there is a net $(f_\gamma)_\gamma$ of finite means such that $\|s\cdot f_\gamma-f_\gamma\|_1\longrightarrow0$ \cite{day, nam}. By a similar argument, we have the following result for $\varphi$-amenability of $S$, where $\varphi\in\Delta_S(S)$.
\begin{thm}\label{fs1}
Let $S$ be a topological semigroup and $\varphi\in\Delta_S(S)$. Then $S$ is left $\varphi$-amenable if and only if there is a bounded net $(f_\alpha)_{\alpha\in I}\subseteq\ell^1(S)$ such that $\|s\cdot f_\alpha-\varphi(s)f_\alpha\|_1\longrightarrow0$ and its $w^*$-limit on $\varphi$ is 1.
\end{thm}
\begin{proof}
Assume that $S$ is left $\varphi$-amenable and  $m$ is a bounded linear functional that is obtained in the Theorem \ref{t1}. Thus, there is a net $(f_\alpha)_{\alpha\in I}\subseteq\ell^1(S)$ such that $w^*$-converges to $m$ and $\|f_\alpha\|_1\leq\|m\|_\infty$. Then
\begin{equation}\label{e1n}
    f(s\cdot f_\alpha-\varphi(s)f_\alpha)=f_\alpha(f\cdot s)-\varphi(s)f_\alpha(f)\longrightarrow m(f\cdot s)-\varphi(s)m(f)=0,
\end{equation}
for all $f\in RUC(S)$ and $s\in S$.  Consider the product space $\ell^1(S)^S$ that is a locally convex linear topological space with the product of the norm topologies. Now define a linear map $T:\ell^1(S)\longrightarrow \ell^1(S)^S$ by $T(g)=(s\cdot g-\varphi(s)g)_{s\in S}$. Thus, if $S$ is left $\varphi$-amenable, then $0$ is in the weak closure of $T$ on the set of finite means such as $f\in\ell^1(S)$. Since the set of finite means is convex in $\ell^1(S)$ and $\ell^1(S)^S$ is locally convex, $T$ on this set is convex. This implies that the weak closure of $T$ on the set of finite means equals the closure of it in the given topology on $\ell^1(S)^S$, that is, the product of the norm topologies. Thus, there is a net $(f_\alpha)_{\alpha\in I}\subseteq\ell^1(S)$ such that $\|s\cdot f_\alpha-\varphi(s)f_\alpha\|_1\longrightarrow0$ and $f_\alpha\stackrel{w^*}{\longrightarrow}m$.

Conversely, let there is a net $(f_\alpha)_{\alpha\in I}\subseteq\ell^1(S)$ such that $w^*$-converges to an element of $RUC(S)^*$ namely $m$ such that $m(\varphi)=1$ and $\|s\cdot f_\alpha-\varphi(s)f_\alpha\|_1\longrightarrow0$, for every $s\in S$. This means that $s\cdot f_\alpha-\varphi(s)f_\alpha\longrightarrow0$ in the weak topology. Thus, similar to \eqref{e1n}, we have $m(f\cdot s)=\varphi(s)m(f)$, for all $f\in RUC(S)$ and $s\in S$. This shows that $S$ is $\varphi$-amenable.
\end{proof}
%%%%%%%%%%%%%%%%%%%%%%%%%%%%%%%%%%%%%%%%
%%%%%%%%%%%%%%%%%%%%%%%%%%%%%%%%%%%%%%%%
%%%%%%%%%%%%%%%%%%%%%%%%%%%%%%%%%%%%%%%%
%%%%%%%%%%%%%%%%%%%%%%%%%%%%%%%%%%%%%%%%
\section{Some Hereditary Properties}
%%%%%%%%%%%%%%%%%%%%%%%%%%%%%%%%%%%%%%%%
%%%%%%%%%%%%%%%%%%%%%%%%%%%%%%%%%%%%%%%%
%%%%%%%%%%%%%%%%%%%%%%%%%%%%%%%%%%%%%%%%
%%%%%%%%%%%%%%%%%%%%%%%%%%%%%%%%%%%%%%%%

This section deals with the stability properties of $\varphi$-amenability of topological semigroups and groups. We prove these properties via Theorem \ref{t1}.
\begin{prop}\label{p1}
Let $S$, $T$ be semigroups, $\theta:S\longrightarrow T$ be a continuous and onto semigroups homomorphism. Let $\psi\in\Delta_T(T)$, and let $S$ be left \emph{(}right\emph{)} $(\psi\circ\theta)$-amenable. Then $T$ is left \emph{(}right\emph{)} $\psi$-amenable.
\end{prop}
\begin{proof}
Let $X\in~_\psi\mathcal{M}^T$ and $D:T\longrightarrow X^*$ be a bounded derivation. Then we can see $X$ as an element of $_{\psi\circ\theta}\mathcal{M}^S$ by the following actions
\begin{equation*}
    s\cdot x=\theta(s)\cdot x=\psi(\theta(s))x, ~\emph{\emph{and}}~x\cdot s=x\cdot\theta(s),
\end{equation*}
for all $s\in S$ and $x\in X$. Define $D\circ\theta:S\longrightarrow X^*$. Obviously, $D\circ\theta$ is a bounded derivation. Thus, there exists $\mathfrak{x}\in X^*$ such that $$(D\circ\theta)(s)=s\cdot\mathfrak{x}-\psi(\theta(s))\mathfrak{x},$$ for all $s\in S$. Since $\theta$ is onto, we have
$$D(t)=t\cdot\mathfrak{x}-\psi(t)\mathfrak{x},$$
 for all $t\in T$. Similarly, we can prove if $S$ is left $\psi\circ\theta$-amenable, then $T$ is left $\psi$-amenable.
\end{proof}
\begin{cor}
Let $S$ be a topological semigroup,  $L$ be a closed ideal in $S$ and $\varphi\in\Delta_S(S)$ such that $\varphi|_L\neq0$. If $S$ is $\varphi$-amenable, then $L$ is $\varphi|_L$-amenable.
\end{cor}
\begin{cor}
Let $G$ be a  locally compact group,  $H$ be a closed normal subgroup of $G$ and $\varphi\in\Delta_{G/H}(G/H)$. If $G$ is $(\varphi\circ\theta)$-amenable, where $\theta:G\longrightarrow G/H$ is the canonical homomorphism, then $G/H$ is $\varphi$-amenable.
\end{cor}

It is well-known that the quotient group of an amenable group $G$ by a closed normal subgroup $H$ is amenable and moreover $H$ is amenable as a group. By these facts, Theorem \ref{t3} and Proposition \ref{p1}, we have the following result:
 \begin{cor}
Let $G$ be a  locally compact group and $H$ be a closed normal subgroup of $G$. Then $G$ is character amenable if and only if $H$ and $G/H$ are character amenable.
\end{cor}

Let $S$ and $T$ be semigroups. Then $S\times T$ is a semigroup with the operation $$(s_1,t_1)(s_2,t_2)=(s_1s_2,t_1t_2),$$ for all $s_1,s_2\in S$ and $t_1,t_2\in T$.  Define $\pi_S:S\times T\longrightarrow S$ and  $\pi_T:S\times T\longrightarrow T$ by $\pi_S(s,t)=s$ and $\pi_T(s,t)=t$, respectively, for all $s\in S$ and $t\in T$. Clearly, both $\pi_S$ and $\pi_T$ are continuous and onto semigroups homomorphisms.

\begin{thm}\label{t14}
Let $S$ and $T$ be two topological  semigroups. If $S\times T$ is left \emph{(}right\emph{)} character amenable, then $S$  and $T$ are left \emph{(}right\emph{)} character amenable.
\end{thm}
\begin{proof}
Let $S\times T$ be left (right) character amenable. Let $\varphi\in\Delta_S(S)$ be an arbitrary and $\pi_S$ be as above. Then $\varphi\circ \pi_S\in \Delta_{S\times T}(S\times T)$ and Proposition \ref{p1} implies that $S$ is left (right) $\varphi$-amenable.  This shows that $S$ is character amenable, because  $\varphi$ was arbitrary. Similarly one can see that $T$ is character amenable.
\end{proof}

We consider the converse of the above Theorem in the special case as follows:
\begin{thm}
Let $S$, $T$ be two topological  semigroups,  $\varphi\in\Delta_S(s)$ and $\psi\in\Delta_T(T)$. If $S$ is left \emph{(}right\emph{)} $\varphi$-amenable and $T$ is left \emph{(}right\emph{)} $\psi$-amenable, then $S\times T$ is left \emph{(}right\emph{)} $(\varphi,\psi)$-amenable.
\end{thm}
\begin{proof}
Suppose that $S$ is left $\varphi$-amenable, $T$ is left $\psi$-amenable,  $m_S$ and $m_T$ are the bounded functionals obtained from Theorem \ref{t1}. For each $f\in RUC(S\times T)$ and $(s,t)\in S\times T$, we can define $f_s\in RUC(T)$ and $g\in RUC(S)$ as follows
$$f_s(t)=f(s,t)\hspace{1cm}\emph{\emph{and}}\hspace{1cm}g(s)=m_T(f_s(t)).$$

Now, define $m$ on $RUC(S\times T)$ by $m(f)=m_S(g)$, for all $f\in RUC(S\times T)$. Then
\begin{eqnarray*}
  m(f\cdot(s,t)) &=& m_S[m_T((f\cdot (s,t))(x,y))]=m_S[m_T(f(sx,ty))] \\
   &=& m_S[m_T(f_{sx}(ty))]=m_S[m_T((f_{sx}\cdot t)(y))]  \\
   &=&  m_S[\psi(t)m_T(f_{sx}(y))]=\psi(t) m_S[g(sx)]\\
   &=& \psi(t) m_S[(g\cdot s)(x)] =\varphi(s)\psi(t) m_S[g(x)] \\
   &=&\varphi(s)\psi(t) m_S[m_T(f(x,y))]\\
   &=&\varphi(s)\psi(t) m(f)
\end{eqnarray*}
for all $f\in RUC(S\times T)$ and $(s,t)\in S\times T$. Clearly, $(\varphi,\psi)\in RUC(S\times T)$ and the above obtained result follows that $m\left((\varphi,\psi)\right)=1$. Thus, Theorem \ref{t1} implies that $S\times T$ is left $(\varphi,\psi)$-amenable.
\end{proof}

An involution on a topological semigroup $S$ is a map $*$ from $S$ into $S$ such that,
 the images of $s, t\in S$ are denoted by $s^*$ and $t^*$, respectively,  $s = (s^*)^*$,  $(st)^* = t^* s^*$ and $*$ is a continuous map; see  \cite{bb, bi} for more results related to topological semigroups with involution and characters on them.  Let $f\in LUC(S)$ or $RUC(S)$, we set $\widetilde{f}(s)=f(s^*)$, for all $s\in S$.
\begin{thm}
Let $S$  be a discrete  semigroups with involution $*$ and $\varphi\in\Delta_S(S)$. If $S$ is left \emph{(}right\emph{)} $\widetilde{\varphi}$-amenable, then $S$ is right \emph{(}left\emph{)} $\varphi$-amenable.
\end{thm}
\begin{proof}
Suppose that $S$ is left (right) $\widetilde{\varphi}$-amenable. Then Theorem \ref{t1} implies that there is a bounded linear functional $m$ in $\ell^\infty(S)^*$ such $m(\widetilde{\varphi})=1$ and $m(f\cdot s^*)=\widetilde{\varphi}(s^*)m(f)$, for all $f\in \ell^\infty(S)$. Let $f\in \ell^\infty(S)$ and define $m'(f)=m(\widetilde{f})$. Since the mapping $f\longmapsto\widetilde{f}$ is linear, $m'$ is linear and bounded. Furthermore, $m'(\varphi)=m(\widetilde{\varphi})=1$ and $m'(f)\geq0$, for all $f\in \ell^\infty(S)$.

Moreover, for all $f\in \ell^\infty(S)$ and $s,t\in S$, we have
\begin{eqnarray*}
 (s\cdot\widetilde{f})(t)  &=& (s\cdot f)(t^*)=f(t^*s)=f\left((s^*t)^*\right) \\
   &=& \widetilde{f}(s^*t)=(\widetilde{f}\cdot s^*)(t).
\end{eqnarray*}

This shows that $s\cdot\widetilde{f}=\widetilde{f}\cdot s^*$ for all $f\in \ell^\infty(S)$  and $s\in S$. Then
\begin{eqnarray*}
  m'(s\cdot f) &=& m(s\cdot \widetilde{f})=m(\widetilde{f}\cdot s^*)=\widetilde{\varphi}(s^*)m(\widetilde{f}) \\
   &=&  \varphi(s)m'(f),
\end{eqnarray*}
for all $f\in \ell^\infty(S)$  and $s\in S$. Thus, $S$ is right $\varphi$-amenable.
\end{proof}

Assume that $S$ is a semigroup and $\varphi\in\Delta_S(S)$. Let $T$ be a subsemigroup of $S$ and $S$ is left (right) $\varphi$-amenable. We denote $\varphi$ by $\varphi|_T$ on $T$, clearly, it is a character on $T$, but, maybe $T$ is not left (right) $\varphi_T$-amenable. In other words, there is a subsemigroup of $\varphi$-amenable semigroup is not $\varphi$-amenable. Moreover, there is a subsemigroup $T$ of semigroup $S$ and $\varphi\in\Delta_T(T)$ such that $T$ is $\varphi$-amenable and $\widetilde{\varphi}$-amenability of $S$ does not sense, where $\widetilde{\varphi}$ is the extension of $\varphi$ on $S$. The following example shows the above statements are true.
\begin{ex}
\begin{itemize}
  \item[(i)] Let $S$ be a  semigroup without zero element $o$. Following \cite{dls}, we denote the semigroup formed by adjoining $o$ to $S$ by $S^o$ and $S$ becomes a subsemigroup of $S^o$. Then the only  character on $S^o$ is $1_{S^o}\in\Delta_{S^o}(S^o)$. Let $S$   be not left $1_S$-amenable. Define $m(f)=f(o)$, for all $f\in RUC(S^o)$. Thus, $S^o$ is left $1_{S^o}$-amenable.
  \item[(ii)] Let $S$ be a semigroup without zero element $o$ and $\varphi\in\Delta_S(S)$ such that $1_S\neq \varphi$. If $S$ is $\varphi$-amenable, then according to $\mathrm{(i)}$ and by this fact that $\varphi$ has not any extension such as $\widetilde{\varphi}$ on $S^o$, $S^o$ is not $\widetilde{\varphi}$-amenable.
\end{itemize}
\end{ex}
\begin{defn}
Let $S$ be a topological semigroup, $T$ be a  right  thick susbemigroup of $S$ and $\varphi\in\Delta_S(S)$. We say that $T$ is strongly left $\varphi|_T$-amenable if there is a bounded linear functional $m$ on $LUC(T)$ such that $\mathrm{(i)}$ $m(\varphi|_T)=1$ and $\mathrm{(ii)}$  $m(s\cdot f)=\varphi(s)m(f)$, for all  $f\in LUC(T)$, $s\in S$. Similarly, one can define the strongly right $\varphi|_T$-amenability for the left thick susbemigroup $T$ of $S$.
\end{defn}

\begin{thm}\label{sub}
Let $S$ be a topological semigroup, $T$ be a left (right)  thick susbemigroup of $S$ and $\varphi\in\Delta_S(S)$. Then $T$ is strongly right (left) $\varphi_T$-amenable if and only if $S$ is right (left) $\varphi$-amenable.
\end{thm}
\begin{proof}
We prove the right case and the left case is similar.
Assume that $T$ is  strongly right $\varphi|_T$-amenable. Define $\Phi:LUC(S)\longrightarrow LUC(T)$ by $\Phi(f)=f|_T$. Clearly, $\Phi$ is a bounded linear map and consider $\Phi^*:LUC(T)^*\longrightarrow LUC(S)^*$. By Theorem \ref{t1}, there is a bounded linear functional $m$ in $LUC(T)^*$ such that $m(\varphi|_T)=1$ and $m(t\cdot f)=\varphi|_T(t)m(f)$, for all $f\in LUC(T)$ and $t\in T$. We claim that $\mathbf{m}=\Phi^*(m)$ is a bounded linear functional for $S$ that satisfies  the condition (ii) of Theorem \ref{t1}. Since $T$ is left thick, for all  $f\in LUC(S)$, $s\in S$ and $t\in T$, we have
$$\Phi(s\cdot f)(t)=(s\cdot f)|_T(t)=f|_T(ts)=(s\cdot\Phi(f))(t).$$

This implies that $\Phi(s\cdot f)=s\cdot\Phi(f)$, for all  $f\in LUC(S)$ and $s\in S$. Then
\begin{eqnarray}\label{e1sub}
\nonumber
   \mathbf{m}(s\cdot f)&=&\Phi^*(m)(s\cdot f)=m(\Phi(s\cdot f))=m(s\cdot\Phi(f))=\varphi(s)m(\Phi(f))\\
   &=&\varphi(s)\mathbf{m}(f),
\end{eqnarray}
for all  $f\in LUC(S)$ and $s\in S$. Moreover,
\begin{equation}\label{e2sub}
    \mathbf{m}(\varphi)=\Phi^*(m)(\varphi)=m(\Phi(\varphi))=m(\varphi_T)=1.
\end{equation}

The relations \eqref{e1sub} and \eqref{e2sub} follow that $S$ is right $\varphi$-amenable.

Consider the canonical bounded linear map $\Psi:LUC(T)\longrightarrow LUC(S)$ such that $\Psi(f)|_{S\setminus T}=0$, for all $\varphi_T\neq f\in LUC(T)$ and $\Psi(\varphi_T)=\varphi$. Let $\mathbf{m}$ be a linear functional defined on $LUC(S)$ that satisfies  Theorem \ref{t1}. We now show that $m=\Psi^*(\mathbf{m})$ is a $\varphi$-mean for $T$.

For an arbitrary $f\in LUC(T)$, $s\in S$ and $t\in T$,
$$\left(\Psi(s\cdot f)-s\cdot\Psi(f)\right)(t)=(s\cdot f)(t)-\Psi(f)(ts)=f(ts)-f(ts)=0.$$

Hence, $\left(\Psi(s\cdot f)-s\cdot\Psi(f)\right)|_T=0$, and
$$|(s\cdot f)(t)-s\cdot \Psi(f)(t)|\leq\|\Psi(s\cdot f)-s\cdot\Psi(f)\|_\infty\chi_{S\setminus T},$$
where $\chi_{S\setminus T}$ is the characteristic function on $S\setminus T$. Thus,
\begin{equation}\label{e3sub}
    \mathbf{m}(\Psi(s\cdot f))=\mathbf{m}(s\cdot\Psi(f))\hspace{1cm}(f\in LUC(T),\ s\in S).
\end{equation}

Therefore,
\begin{eqnarray}\label{e4sub}
\nonumber
m(s\cdot f)&=&\Psi^*(\mathbf{m})(s\cdot f)=\mathbf{m}(\Psi(s\cdot f))=\mathbf{m}(s\cdot\Psi(f))=\varphi(s)\mathbf{m}(\Psi(f))\\
&=&\varphi(s)m(f),
\end{eqnarray}
and
\begin{equation}\label{e5sub}
  m(\varphi_T)=\Psi^*(\mathbf{m})(\varphi_T) =\mathbf{m}(\Psi(\varphi_T))=\mathbf{m}(\varphi)=1,
\end{equation}
for all $f\in LUC(T)$ and $s\in S$. Thus, $T$ is strongly right $\varphi_T$-amenable.
\end{proof}

We finish this section with the following result:
\begin{prop}
Let $\{S_\alpha\}_{\alpha\in I}$ be a family of closed subsemigroups of topological semigroup $S$ such that $S_\alpha$ is left (right) $\varphi_\alpha$-amenable, for each $\alpha\in I$, where $\varphi_\alpha\in\Delta_{S_\alpha}(S_\alpha)$. Let the following conditions hold:
\begin{itemize}
  \item[(i)] for every $S_\alpha$, $S_\beta$ that are left (right) $\varphi_\alpha$-amenable and $\varphi_\beta$-amenable, respectively, there is a $S_\gamma$ such that $S_\alpha\cup S_\beta\subseteq S_\gamma$ and $S_\gamma$ is $\varphi_\gamma$-amenable.
  \item[(ii)] $S=\overline{\bigcup_{\alpha\in I} S_\alpha}$.
\end{itemize}

Let $\varphi$ be a function on $S$ such that
$$\varphi(st)=\left\{
  \begin{array}{ll}
   \varphi_\alpha(st) &\text{if} \ s,t\in S_\alpha \\
  \varphi_\gamma(st)  & \text{if} \ s\in S_\alpha, t\in S_\beta\ \text{and}\ S_\alpha\cup S_\beta\subseteq S_\gamma
  \end{array}
\right.$$

 Then $\varphi$ is a character on $S$ and $S$ is left (right) $\varphi$-amenable.
\end{prop}
\begin{proof}
Clearly, $\varphi$ is a character on $S$. Let $m_\alpha$ be a bounded linear functional on $RUC(S_\alpha)$, for every $\alpha\in I$, that satisfies  Theorem \ref{t1}. Define
$$m_\alpha'(f)=m_\alpha(f|_{S_\alpha}),$$
for every $f\in RUC(S)$. Let $M_\alpha$ be the $w^*$-closed set of all bounded linear functionals on $RUC(S)$ such as $m$ such that $m(\varphi_\alpha)=1$ and $m(f\cdot s)=\varphi_\alpha(s) m(f)$, for all $f\in RUC(S)$ and $s\in S_\alpha$. According to the definition of $m_\alpha'$, it belongs to $M_\alpha$. This means that $M_\alpha$ is not empty and it is obvious that the augmentation character $1_S$ is in $RUC(S)$. These imply that $\bigcap_{\alpha\in I}M_\alpha$ is not empty. Now, let $\mathbf{m}\in \bigcap_{\alpha\in I}M_\alpha$. Then the case (i) and definition of $\varphi$ together imply that $\mathbf{m}(\varphi)=1$. Moreover, the cases (i), (ii) and definition of $\varphi$ together imply that for all $f\in RUC(S)$ and $s\in\bigcup_{\alpha\in I}S_\alpha$, there exists $\gamma\in I$ such that  $s\in S_\gamma$ and %$m_\gamma$ on $RUC(S_\gamma)$ such that
\begin{equation}\label{e1fh}
\mathbf{m}(f\cdot s)=\varphi_\gamma(s)m_\gamma(f|_{S_\gamma})=\varphi(s)\mathbf{m}(f).
\end{equation}

Since $f\in RUC(S)$, for every $s\in S$ and  $\varepsilon>0$, there is a net $(t_\beta)_\beta\subseteq\bigcup_{\alpha\in I}S_\alpha$ such that $\|(f\cdot s)-(f\cdot t_\beta)\|_\infty<\varepsilon/2\|\mathbf{m}\|$ and $|\varphi(t_\beta)-\varphi(s)|<\varepsilon/2\|\mathbf{m}\|$. Then \eqref{e1fh} implies that
\begin{eqnarray*}
  |\mathbf{m}(f\cdot s)-\varphi(s)\mathbf{m}(f)| &=&|\mathbf{m}(f\cdot s)-\mathbf{m}(f\cdot t_\beta)+\mathbf{m}(f\cdot t_\beta)-\varphi(s)\mathbf{m}(f)|  \\
   &=& |\mathbf{m}(f\cdot s)-\mathbf{m}(f\cdot t_\beta)+\varphi(t_\beta)\mathbf{m}(f)-\varphi(s)\mathbf{m}(f)|  \\
   &\leq&|\mathbf{m}(f\cdot s)-\mathbf{m}(f\cdot t_\beta)|+|\varphi(t_\beta)-\varphi(s)||\mathbf{m}(f)|\\
   &\leq&\|\mathbf{m}\| \|(f\cdot s)-(f\cdot t_\beta)\|_\infty+|\varphi(t_\beta)-\varphi(s)|\|\mathbf{m}\|\\
   &<&\varepsilon.
\end{eqnarray*}

Since $\varepsilon$ was arbitrary, $\mathbf{m}(f\cdot s)=\varphi(s)\mathbf{m}(f)$, for all $f\in RUC(S)$ and $s\in S$. This means that $S$ is left $\varphi$-amenable.
\end{proof}

%%%%%%%%%%%%%%%%%%%%%%%%%%%%%%%%%%%%%%
%%%%%%%%%%%%%%%%%%%%%%%%%%%%%%%%%%%%%%
%%%%%%%%%%%%%%%%%%%%%%%%%%%%%%%%%%%%%%
\section{$\varphi$-Ergodic Properties}
%%%%%%%%%%%%%%%%%%%%%%%%%%%%%%%%%%%%%%
%%%%%%%%%%%%%%%%%%%%%%%%%%%%%%%%%%%%%%
%%%%%%%%%%%%%%%%%%%%%%%%%%%%%%%%%%%%%%

Let $S$ be a topological semigroup, $X$ be a Banach space and $\mathbf{B}(X)$ be the Banach space of all bounded operators on $X$. An antirepresentation  of $S$ on $X$ is a function $F:s\longmapsto F_s$ such that $F_{st}=F_tF_s$, for each $s,t\in S$. For each $s\in S$, define $L=\ell(s)=\ell_s$ in $\mathbf{B}(C(S))$ by $Lf(t)=f(st)$, for all $t\in S$ and $f\in C(S)$. The function $\ell$ is the antirepresentation  of $S$ on $X$.

 Let $\varphi\in\Delta_S(S)$, Similar to \cite{day1}, we define the following sets that we work on them in this section:
\begin{itemize}
  \item[] $P(\varphi)=\{s\in S:\varphi(s)=1\}$,
  \item[] $M_0^\varphi(\ell)=\{x\in X:(\ell_s-I)x=0,\ s\in P(\varphi)\}$,
  \item[] $M_1^\varphi(\ell)=\text{closed linear hull of }\{(\ell_s-I)x: x\in X, \ s\in P(\varphi)\}$,
  \item[] $M_\varphi(\ell)=M_0^\varphi(\ell_s)+M_1^\varphi(\ell_s)$,
  \item[] $N_x^\varphi(\ell)=\text{the closure of }\{\ell_s(x):s\in P(\varphi)\}$, for every $x\in X$.
\end{itemize}

Now, we generalize the ergodicity of the  antirepresentation $\ell_s$ of $S$ into $X$ as follows:
\begin{defn}\label{de}
Let $S$ be a topological semigroup, $X$ be a Banach space and $\varphi\in\Delta_S(S)$. We say that the antirepresentation $\ell$ from $S$ into $X$ is left $\varphi$-ergodic if there is a bounded net $(B_\delta)_{\delta\in I}$ in $\mathbf{B}(X)$ such that
\begin{itemize}
  \item[($E_1$)] $\lim_\delta B_\delta(\ell_s-I)=0$ in strong operator topology of $\mathbf{B}(X)$, for every $s\in P(\varphi)$.
  \item[($E_2$)] $B_\delta(x)\in N_x^\varphi(\ell_s)$, for each $x\in X$ and $\delta\in I$.
\end{itemize}

Similarly, we call the antirepresentation $\ell$ from $S$ into $X$ is right $\varphi$-ergodic if there is a bounded net $(B_\delta)_{\delta\in I}$ in $\mathbf{B}(X)$ such that satisfies ($E_2$) and the following condition:
\begin{itemize}
  \item[($E_3$)] $\lim_\delta (\ell_s-I)B_\delta=0$ in strong operator topology of $\mathbf{B}(X)$, for every $s\in P(\varphi)$.
\end{itemize}

If the antirepresentation $\ell$ from $S$ into $X$ is right and left $\varphi$-ergodic, we call it $\varphi$-ergodic.
\end{defn}

The following result is the  generalization of the obtained results by Eberlein in \cite{eb} where the proof is similar and we give it proof for clearness:
\begin{thm}\label{te1}
Let $S$ be a topological semigroup, $X$ be a Banach space and $\varphi\in\Delta_S(S)$. Assume that the antirepresentation $\ell$ from $S$ into $X$ is left $\varphi$-ergodic with $(B_\delta)_{\delta\in I}$ in $\mathbf{B}(X)$ that satisfies  the cases ($E_1$) and ($E_2$). Then
\begin{itemize}
  \item[(i)] $B_\delta(x)=x$, for all $x\in M_0^\varphi(\ell)$ and $\delta\in I$.
  \item[(ii)] $B_\delta(x)\longrightarrow0$, for every $x\in M_1^\varphi(\ell)$.
  \item[(iii)] $(B_\delta(x))_{\delta\in I}$ is norm convergent to an element of $M_0^\varphi(\ell)\cap N_x^\varphi(\ell)$.
  \item[(iv)] $M_\varphi(\ell)=M_0^\varphi(\ell)\oplus M_1^\varphi(\ell)$.
  \item[(v)] $\ell_s(M_\varphi(\ell))\subseteq M_\varphi(\ell)$ and $N_x^\varphi(\ell)\subseteq M_\varphi(\ell)$, for all $s\in P(\varphi)$ and $x\in  M_\varphi(\ell)$.
   \item[(vi)]  $B_\delta(M_\varphi(\ell))\subseteq M_\varphi(\ell)$ for all $\delta\in I$.
    \item[(vii)] suppose that $\pi:M_\varphi(\ell)\longrightarrow M_0^\varphi(\ell)$ is a projection associated with the direct sum decomposition (iv), then $B_\delta(x)\longrightarrow \pi(x)$ and $M_0^\varphi(\ell)\cap N_x^\varphi(\ell)=\{\pi x\}$, for all $x\in M_\varphi(\ell)$.
\end{itemize}
\end{thm}
\begin{proof}
(i) If $x\in M_0^\varphi(\ell)$, then $\ell_sx=Ix=x$. This implies that $N_x^\varphi(\ell)=\{x\}$ and consequently, $E_2$ leads  $B_\delta(x)=x$, for every $\delta\in I$.

(ii) Assume that $x\in M_1^\varphi(\ell)$. Then  the case ($E_1$)  together with $(B_\delta)_{\delta\in I}$ is bounded, we have $B_\delta(x)\longrightarrow0$.

(iii) The cases (i) and (ii) together imply  this case.

(iv) The cases (i) and (ii) together imply $M_0^\varphi(\ell)\cap M_1^\varphi(\ell)=\{0\}$ and this means that $M_\varphi(\ell)=M_0^\varphi(\ell)\oplus M_1^\varphi(\ell)$.

(v) For all $s,t\in P(\varphi)$ and $x\in X$, we have
$$\ell_s(\ell_t-I)(x)=(\ell_{ts}-\ell_s)(x)=(\ell_{ts}-I)(x)-(\ell_s-I)(x)\in M_1^\varphi(\ell),$$
because  $ts\in P(\varphi)$. This means that $\ell_s(M_1^\varphi(\ell))\subseteq M_1^\varphi(\ell)$. Then by applying (i) and (iv), we have $\ell_s(M_\varphi(\ell))\subseteq M_\varphi(\ell)$.

For showing that $N_x^\varphi(\ell)\subseteq M_\varphi(\ell)$, for every $x\in  M_\varphi(\ell)$, pick $x\in  M_\varphi(\ell)$ and let $y\in N_x^\varphi(\ell)$. According to the definition of  $N_x^\varphi(\ell)$, there is a net $(s_\alpha)_{\alpha\in J}$ in $P(\varphi)$ such that $\ell_{s_\alpha}(x)\longrightarrow y$. Then (iv) implies that there exist $e\in M_0^\varphi(\ell)$ and $k\in M_1^\varphi(\ell)$ such that $x=e+k$. Since $\ell_s(M_1^\varphi(\ell))\subseteq M_1^\varphi(\ell)$ and  $M_1^\varphi(\ell)$ is closed, by (i) we have
$$y-e=\lim_\alpha\ell_{s_\alpha}(x-e)=\lim_\alpha\ell_{s_\alpha}(k)\in M_1^\varphi(\ell).$$

Again by (iv), we conclude that $N_x^\varphi(\ell)\subseteq M_\varphi(\ell)$, for every  $x\in  M_\varphi(\ell)$.

(vi) Apply (v) and ($E_1$).

(vii) The parts (i) and (ii) imply that $B_\delta(x)\longrightarrow \pi(x)$, for all $x\in M_\varphi(\ell)$. Let $x\in M_\varphi(\ell)$ be arbitrary. Then by the parts (i) and (ii) we have
\begin{equation}\label{e1te1}
   \pi(x)\in M_0^\varphi(\ell)\cap N_x^\varphi(\ell).
\end{equation}

Now, assume that $y\in M_0^\varphi(\ell)\cap N_x^\varphi(\ell)$. Since $y\in N_x^\varphi(\ell)$, there is a net $(s_\alpha)_{\alpha\in J}$ in $P(\varphi)$ such that $\ell_{s_\alpha}(x)\longrightarrow y$. Then
$$y-x=\lim_\alpha (\ell_{s_\alpha}-I)(x)\in  M_1^\varphi(\ell).$$

This shows that $\pi(y-x)=0$. Thus, $\pi(x)=y$ and this completes the proof.
\end{proof}

The above Theorem immediately follows the following result that is a generalization of the obtained results in \cite{day1}.
\begin{prop}\label{pe1}
Let $S$ be a topological semigroup, $X$ be a Banach space and $\varphi\in\Delta_S(S)$. Assume that the antirepresentation $\ell$ from $S$ into $X$ is  $\varphi$-ergodic with $(B_\delta)_{\delta\in I}$ in $\mathbf{B}(X)$ that satisfies  ($E_1$), ($E_2$) and ($E_3$). Then the following statements hold.
\begin{itemize}
  \item[(i)] $M_\varphi(\ell)$ is closed in $X$.
  \item[(ii)] If $N_x^\varphi(\ell)$ is weakly compact, for every $x\in X$, then $M_\varphi(\ell)=X$.
\end{itemize}
\end{prop}
\begin{proof}
(i) Theorem \ref{te1}(iii) implies that $(B_\delta)_{\delta\in I}$ is norm convergent to an element of $M_0^\varphi(\ell)\cap N_x^\varphi(\ell)$, for every $x\in M_\varphi(\ell)$. This follows that $(B_\delta)_{\delta\in I}$ is weakly convergent to an element of $M_0^\varphi(\ell)\cap N_x^\varphi(\ell)$, for every $x\in M_\varphi(\ell)$.

Let $x\in X$ and $B_\delta x\longrightarrow y$ weakly in $X$, for some $y\in X$. We shall show that $y\in M_\varphi(\ell)$.  For any $s\in P(\varphi)$ and $T\in X^*$, we have
\begin{eqnarray*}
  T(\ell_s)(y) &=& \lim_\delta T(\ell_s)(B_\delta(x)) = \lim_\delta [T(\ell_s-I)(B_\delta(x))+T(B_\delta(x))]\\
   &=&  T(y).
\end{eqnarray*}

This means that $\ell_s(y)=y$ and consequently, $y\in  M_0^\varphi(\ell)$. Moreover, for every $x\in X$,  $N_x^\varphi(\ell)$  is convex and norm closed in $X$, so, it is weakly closed and $y\in N_x^\varphi(\ell)$. Thus, $y\in M_0^\varphi(\ell)\cap N_x^\varphi(\ell)$. Then
$$x=y+(x-y)\in  M_0^\varphi(\ell)+ M_1^\varphi(\ell)= M_\varphi(\ell).$$

Hence, $M_\varphi(\ell)$ is closed in $X$.

(ii) Assume that $N_x^\varphi(\ell)$ is weakly compact, for every $x\in X$. Let $y\in X$ be an arbitrary element. The fact $(B_\delta)_{\delta\in I}\in N_x^\varphi(\ell)$ implies that there is a subnet $(B_{\delta_\gamma})$ of $(B_\delta)_{\delta\in I}$ such that $B_{\delta_\gamma}(y)$ is weakly convergent in $N_x^\varphi(\ell)$. Now, if we replace  $(B_\delta)_{\delta\in I}$ by $(B_{\delta_\gamma})$ in the proof of the part (i), then $y\in  M_\varphi(\ell)$. Thus, $X= M_\varphi(\ell)$.
\end{proof}

Let $S$ be a locally compact topological semigroup and  $X$ be a right Banach $S$-module. Then we can see $X$ as a right Banach $\ell^1(S)$-module as follows:
$$xf=\int_Sxs\ \textrm{d}f(s),$$
for all $x\in X$ and $f\in\ell^1(S)$; see \cite[Proposition 5.6]{pet} for more details. Furthermore, for every $T\in X^*$, we define
$$T(xf)=f(Tx),$$
for all $x\in X$ and $f\in\ell^1(S)$.
We now give a relation between left $\varphi$ amenability and left $\varphi$-ergodicity of an antirepresentions as follows:
\begin{thm}\label{tea}
Let $S$ be a locally compact topological semigroup, $X$ be a Banach space, $\varphi\in\Delta_S(S)$ and  $\ell$ be the right action of $S$ on $X$ i.e. $\ell_s(x)=xs$. If $N_x^\varphi(\ell)$ is weakly compact, for every $x\in X$, then the following statements are equivalent:
\begin{itemize}
  \item[(i)] $S$ is left $\varphi$-amenable.
  \item[(ii)] the antirepresention $\ell$ is left $\varphi$-ergodic.
\end{itemize}
\end{thm}
\begin{proof}
 (i)$\Longrightarrow$ (ii) Assume that $\ell$ is an antirepresentaion from $S$ into $X$. Theorem \ref{fs1} implies that there is a bounded net $(f_\alpha)_{\alpha\in I}\subseteq \ell^1(S)$ such that $\|s\cdot f_\alpha-\varphi(s)f_\alpha\|_1\longrightarrow0$ and its $w^*-\lim$ on $\varphi$ is 1. Put $B_\alpha=\ell_{f_\alpha}$, for all $\alpha\in I$. Clearly, $B_\alpha$ is bounded in $\mathbf{B}(X)$. Moreover, by noting that for $s\in S$ and $x\in X$, $x\delta_s=xs$, then
 \begin{eqnarray}\label{etea1}
\nonumber
 \|B_\alpha(\ell_{s}-I)\|   &=&\|B_\alpha(\ell_{\delta_s}-I)\|=\|\ell_{f_\alpha}\ell_{\delta_s}-\ell_{f_\alpha}\|=\|\ell_{\delta_s\ast f_\alpha}-\ell_{f_\alpha}\|=\|\ell_{s\cdot f_\alpha-f_\alpha}\|  \\
    &\leq&  \|\ell\|\|s\cdot f_\alpha-f_\alpha\|_1.
 \end{eqnarray}

 Therefore, \eqref{etea1} implies that
\begin{equation}
   \|B_\alpha(\ell_{s}-I)\|\leq  \|\ell\|\|s\cdot f_\alpha-\varphi(s)f_\alpha\|_1\longrightarrow0,
\end{equation}
for all  $s\in P(\varphi)$. Now, we must show that $B_\alpha$ satisfies  ($E_2$) i.e., $B_\alpha\in N_x^\varphi(\ell)$, for all $x\in X$ and $\alpha\in I$.   Note that $$N_x^\varphi(\ell)\subseteq\overline{ \{\ell_f(x):f\in\ell^1(S)\ \text{such that} f \ \text{is a finite mean}\}}=K.$$

 We claim that $\subseteq$ must be equality in the case that $N_x^\varphi(\ell)$ is weakly compact, for every $x\in X$. Assume towards a contradiction that there esists $f\in\ell^1(S)$ such that $\ell_f\notin N_x^\varphi(\ell)$. Thus, there exist $T\in X^*$ and $r\in\mathbb{R}$ such that
\begin{equation}\label{etea2}
  \mathbf{Re}\  T(\ell_s(x))<r<\textbf{Re}\ T(\ell_{f}(x))
\end{equation}
for every $s\in S$.
 On the other hand,
 \begin{eqnarray}\label{etea2}
 \nonumber
 \textbf{Re}\ T(\ell_{f}(x))&=&\textbf{Re}\ T(\int_Sxs\ \textrm{d}f(s))=\int_S\textbf{Re}\ T(\ell_s(x))\ \textrm{d}f(s)\\
 &<& \mathbf{Re}\  T(\ell_s(x)).
\end{eqnarray}

A contradiction. Thus, $N_x^\varphi(\ell)= K$. This means that $B_\alpha$ satisfies  ($E_2$).

(ii)$\Longrightarrow$ (i) Let $X=RUC(S)$, $\varphi\in\Delta_S(S)$, and let $\ell$ be an antirepresntation from $S$ into $LUC(S)$ such that
$$\ell_s(f)=f\cdot s=\varphi(s)f,$$
for every $s\in S$ and
$$\ell_s(f)=f\cdot s=\varphi(s)f=f,$$
for every $s\in P(\varphi)$.  Then,
$$(\ell_s-I)(\varphi)=\ell_s(\varphi)-\varphi=\varphi\cdot s-\varphi=\varphi(s)\varphi-\varphi=0,$$
for every $s\in P(\varphi)$. This means that $\varphi\in M_0^\varphi(\ell)\subseteq RUC(S)$ and
 Theorem \ref{te1}(iv) follows that $\varphi\notin M_1^\varphi(\ell)\subseteq RUC(S)$. Thus, the Hahn-Banach Theorem implies that there exists $m\in RUC(S)^*$ such that $m(\varphi)=1$ and $m|_{M_1^\varphi(\ell)}=0$. Moreover,  we have
$$m(f\cdot s)=\varphi(s)m(f),$$
for all $s\in S$ and $f\in LUC(S)$. Thus $S$ is left $\varphi$-amenable.
\end{proof}
%%%%%%%%%%%%%%%%%%%%%%%%%%%%%%%%%%%%%%
%%%%%%%%%%%%%%%%%%%%%%%%%%%%%%%%%%%%%%
%%%%%%%%%%%%%%%%%%%%%%%%%%%%%%%%%%%%%%
\section{$\varphi$-Amenability and Fixed Point Property}
%%%%%%%%%%%%%%%%%%%%%%%%%%%%%%%%%%%%%%
%%%%%%%%%%%%%%%%%%%%%%%%%%%%%%%%%%%%%%
%%%%%%%%%%%%%%%%%%%%%%%%%%%%%%%%%%%%%%
Let $S$ be a topological semigroup, $C_r(S)$ be the space of all bounded real-valued functions on $S$ with supremum norm, $X$ be a translation-invariant closed subalgebra of $C_r(S)$ that contains the constant functions,  $Y$ be a compact Hausdorff space and $C_r(Y)$ be the space of all bounded real-valued continuous functions on $Y$, where $C_r(Y)$ has the supremum norm. Assume that $s\longmapsto\lambda_s$ is a representation of $S$ by continuous self-maps of $Y$. For every $y\in Y$, we define $T_y:C_r(Y)\longrightarrow C_r(S)$ by $T_y(h)(s)=h(\lambda_sy)$, for all $s\in S$ and $h\in C_r(Y)$. The representation $\lambda$ is called $D$-representation of $S$, $X$ on $Y$ if $\{y\in Y:\ T_y(C_r(Y))\subseteq X\}$ is dense in $Y$. If $s\longmapsto\lambda_s$ is a representation of $S$ by continuous affine self-maps of $Y$, then it is called $D$-representation of $S$, $X$ on $Y$ by continuous affine maps if $\{y\in Y:\ T_y(A(Y))\subseteq X\}$ is dense in $Y$.

The existence of a common fixed point of the family $\lambda_S$, whenever $X$ has a left invariant mean   is considered by Argabright in \cite{ar} and Mitchell in \cite{mit}. Indeed, Day's fixed point Theorem for topological semigroups is explained and investigated by Argabright and Mitchell works. Namioka in \cite{nam2} showed that $LUC(S)$ is a translation-invariant subspace of $C_r(S)$ \cite[Lemma 2]{nam2} and similarly, one can see that $RUC(S)$ is a translation-invariant subspace of $C_r(S)$. We recall the following result:
\begin{thm}\label{mit}\cite[Theorem 1]{mit}
Let $S$ be a topological semigroup. Then the following assertions equivalent:
\begin{itemize}
  \item[(i)] $RUC(S)$ has a multiplicative left invariant mean.
  \item[(ii)] whenever $S$ acts on a compact Hausdorg space $Y$, where the map $S\times Y\longrightarrow Y$ is jointly continuous, then $Y$ contains a common fixed point of $S$.
\end{itemize}
\end{thm}

Let $S$ be a topological semigroup, $\varphi\in\Delta_S(S)$ and $P(\varphi)$ be as defined in the previous section. Clearly, $P(\varphi)$ is a topological subsemigroup of $S$ such $\varphi|_{P(\varphi)}=1$. Then by replacing $S$ by $P(\varphi)$ in Theorem \ref{mit}, we have the following result that we give its proof because we use some of the obtained results in the proof, for the last result of this paper:
\begin{cor}\label{fpt}
Let $S$ be a topological semigroup and $\varphi\in\Delta_S(S)$. Then the following assertions equivalent:
\begin{itemize}
  \item[(i)] There is a bounded linear functional $m\in RUC(P(\varphi))^*$ such that $m(\varphi)=1$ and $m(f\cdot s)=m(f)$, for all $f\in RUC(P(\varphi))$ and $s\in P(\varphi)$.
  \item[(ii)] $Y$ contains a common fixed point of $P(\varphi)$ for continuous affine actions of $P(\varphi)$ on compact convex sets of a locally convex linear topological space.
\end{itemize}
\end{cor}
\begin{proof}
(i)$\Longrightarrow$(ii) Let $Y$ be compact convex set of locally convex linear topological space $X$. For each $y\in Y$, we define $T_y:C_r(Y)\longrightarrow C_r(P(\varphi))$ and $R_y:S\longrightarrow Y$ by $(T_yh)(s)=h(ys)$ and $R_y(s)=ys$, for all $h\in C_r(Y)$ and $s\in P(\varphi)$. Then
$$(T_yh)(s)=h(ys)=(hR_y)(s),$$
for all $h\in C_r(Y)$ and $s\in P(\varphi)$. Thus, $T_yh\in C(S)$, because  $R_y(s)$  is continuous. Now, set $T_yh=f$, then
$$r_sf(t)=f(ts)=(T_yh)(ts)=h(yts),$$
for all $s,t\in P(\varphi)$. We claim that $f\in RUC(P(\varphi))$. Assume towards a contradiction that $f\notin RUC(S)$ i.e., there exist $s\in P(\varphi)$ and a net $(s_\alpha)_{\alpha\in I}\subseteq P(\varphi)$ such that $s_\alpha\longrightarrow s$ but $r_{s_\alpha}f$ does not convergent uniformly to $r_sf$. This means that there is a positive number $\beta$ and a net $(t_\alpha)_{\alpha\in I}\subseteq P(\varphi)$ such that
$$|h(yt_\alpha s_\alpha)-h(y(t_\alpha s))|\geq\beta,$$
for all $\alpha\in I$. Denote $yt_\alpha$ by $y_\alpha$. Since $Y$ is compact,  $(y_\alpha)_{\alpha\in I}$ has a subnet such as $(y_\gamma)_{\gamma\in I}$ such that converges to some $y'\in Y$. By continuity of $h$ and joint continuity of action of $P(\varphi)$ on $Y$, we have
$$0<\beta\leq\lim_\gamma|h(y_\gamma s_\gamma)-h(y_\gamma s)|=|h(y' s)-h(y' s)|=0.$$

This is a contradiction and so, $f\in RUC(P(\varphi))$. This leads to that $T_yC(Y)\subseteq RUC(P(\varphi))$ and consequently, $T_yA(Y)\subseteq RUC(P(\varphi))$. Let $m\in RUC(S)^*$ such that $m(\varphi)=1$ and $m(f\cdot s)=\varphi(s)m(f)$, for all $f\in RUC(S)$ and $s\in S$. Now; without loss of generality, consider $T_y:A(Y)\longrightarrow RUC(P(\varphi))$. Then the adjoint of $T_y$ is $T_y^*:RUC(P(\varphi))^*\longrightarrow A(Y)^*$. Thus, there exists $y'\in Y$ such that $T_y^*m(h)=h(y')$, for every $h\in A(Y)$.

For $s\in P(\varphi)$, define $\Gamma_s:A(Y)\longrightarrow A(Y)$ by $\Gamma_sh(y)=h(sy)$, for every $h\in A(Y)$. Then
\begin{eqnarray}\label{fpte1}
\nonumber
  (T_y\Gamma_sh)(t)&=& h(yst)=T_yh(st) \\
   &=&  (T_yh\cdot s)(t),
\end{eqnarray}
for all  $h\in A(Y)$ and $t\in P(\varphi)$. Moreover, by \eqref{fpte1} we have
\begin{eqnarray}\label{fpte2}
\nonumber
  h(sy') &=& \Gamma_sh(y')=\Gamma_s(T_y^*m(h))=m(T_y(\Gamma_sh))=m(T_yh\cdot s) \\
  \nonumber
   &=&  \varphi(s)m(T_yh)=m(T_yh)=T_y^*m(h)\\
   &=&h(y'),
\end{eqnarray}
for all  $h\in A(Y)$ and $s\in P(\varphi)$. This shows that $sy'=y'$.

(ii)$\Longrightarrow$(i) Set $X=RUC(P(\varphi))^*$ with $w^*$-topology and suppose that $Y$ is the compact convex set of all bounded linear functionals on $RUC(P(\varphi))$ such that, for each $m\in Y$,  $m(\varphi)=1$. Define the affine action $T:P(\varphi)\times Y\longrightarrow Y$ by $T_s\mu=l_s^*\mu$, for all $s\in P(\varphi)$ and $\mu\in Y$, where $l_s^*$ is the adjoint of $l_s$. It is easy to check that $T$ is a uniformly continuous affine action on $Y$. Thus, (ii) implies that there exists  $m\in Y$ that is fixed under the affine action $T$  of $P(\varphi)$ such that $l_s^*m=s\cdot m=m$. This means that $s\cdot m(f)=m(f\cdot s)=m(f)$. Thus, (ii) implies (i).
\end{proof}
\begin{remark}\label{rem1}
Note that if a topological semigroup $S$ is $\varphi$-amenable, where $\varphi\in\Delta_S(S)$, then the condition $\mathrm{(ii)}$ holds by the same reasons in the proof of  the above Corollary, but, we do not know the converse part holds or not?
\end{remark}
Similar to Corollary \ref{fpt}, one can consider the conditions weakly right uniformly continuous functions on $S$ instead of right uniformly continuous functions, we do not consider this case in this paper.

Suppose that $E$ is a separated locally convex space and $X$ is a subset of $E$ containing an $n$-dimensional subspace. Let $S$ be a topological semigroup and $\mathcal{T}=\{T_s:s\in S\}$ be a representation of  $S$ as continuous linear transformations from $E$ into $E$ such that $T_s(L)$ is an $n$-dimensional subspace contained in $X$ whenever $L$ is an $n$-dimensional subspace contained in $X$, and there exists a closed $\mathcal{T}$-invariant subspace $H$ in $E$ of codimension $n$ with the property that $x + H \cap X$ is compact and convex, for each $x \in E$. If $S$ is left amenable, then exists an $n$-dimensional subspace $L_0$, contained in $X$ such that $T_s(L_0) = L_0$, for all $s\in S$ \cite{fan}. Lau generalized the above result and introduced the $\mathcal{P}(n)$ condition for topological semigroups and proved that the condition $\mathcal{P}(1)$ implies the left amenability of topological semigroups \cite{lau2}. Let $X\subseteq E$, following  \cite{lau2} by $\mathcal{L}_n(X)$ we mean  all $n$-dimensional subspaces of $E$ contained in $X$. We now write $\mathcal{P}(n)$ condition as follows:

$\mathcal{P}(n)$: Let $S$ be a topological  semigroup and $\mathcal{T}=\{T_s:s\in P(\varphi)\}$ be a representation of  $P(\varphi)$ as linear operators from $E$ into $E$ jointly continuous on compact convex subsets of $E$. Let $X$ be a subset of $E$ such that there is a
closed $\mathcal{T}$-invariant subspace $H$ in $E$ of codimension $n$ with the property that $x + H \cap X$ is compact and convex, for each $x \in E$. If $\mathcal{L}_n(X)$ is non-empty and $\mathcal{T}$-invariant, then there exists $L_0\in \mathcal{L}_n(X)$ such that $T_s(L_0) = L_0$, for each $s\in P(\varphi)$.

Similar to proof of  \cite[Theorem 1]{lau2} and by Corollary \ref{fpt}, we have the following result:
\begin{cor}
Let $S$ be a topological semigroup and $\varphi\in\Delta_S(S)$.
\begin{itemize}
  \item[(i)] If $S$ is left $\varphi$-amenable, then $S$ satisfies $\mathcal{P}(n)$, for each positive integer $n$.
  \item[(ii)] If $S$ satisfies $\mathcal{P}(1)$, then $S$ is left $\varphi$-amenable.
\end{itemize}
\end{cor}
Let $S$ be a topological  semigroup and $\ell:S\times RUC(S)^*\longrightarrow RUC(S)^*$ be the left action of $S$ on $RUC(S)^*$ defined by $(s,m)\longmapsto \ell_sm$, for all $m\in RUC(S)^*$ and $s\in S$  such that $\ell_sI=I$ and $\ell_s\ell_{s'}=\ell_{s's}$, for all $s,s'\in P(\varphi)$, where $I$ is the identity element of $RUC(S)^*$. Thus, $\ell$ on $P(\varphi)$ is an antirepresentation. Suppose that $X\subseteq RUC(S)^*$ is a subspace that contains $I$ and all elements of $RUC(S)^*$ such as $m$ such that $m(\varphi)=1$. A mean $\mu$ on $X$ is called $P(\varphi)$-invariant under $\ell$ if $\mu\ell_s=\mu$, for every $s\in P(\varphi)$ and we say that $\mu$ on $X$ is $S_\varphi$-invariant under $\ell$ if $\mu\ell_s=\varphi(s)\mu$, for every $s\in S$. Clearly, every $S_\varphi$-invariant mean under $\ell$ is $P(\varphi)$-invariant.
\begin{defn}
Let $S$ be a topological  semigroup. We say that $S$ has the Hahn-Banach Theorem Property if, for each continuous left action  $\ell:S\times RUC(S)^*\longrightarrow RUC(S)^*$ and every  $P(\varphi)$-invariant subspace $X$ of $RUC(S)^*$ that contains $I$ and all elements of $RUC(S)^*$ such as $m$ such that $m(\varphi)=1$, every $P(\varphi)$-invariant mean $\mu$ on $X$ can be extended to a $S_\varphi$-invariant mean $\widetilde{\mu}$ on $RUC(S)^*$.
\end{defn}

The Hahn-Banach Property of semigroups is studied by Silverman \cite{sil} and for a special semigroups namely adjoint
 semigroups studied by van Neerven \cite{van}. By the following, we characterize left $\varphi$-amenability of $S$:
\begin{thm}\label{hbt}
Let $S$ be a topological  semigroup and $\varphi\in\Delta_S(S)$, the following assertions are equivalent:
\begin{itemize}
  \item[(i)] $S$ is left $\varphi$-amenable.
  \item[(ii)] $S$ has the Hahn-Banach Theorem Property.
    \item[(iii)] for any Banach $S$-submodule $Y$ of $X$, each linear functional in $\bigcap_{s\in S}\{y^*\in Y^*:~ s\cdot y^*=\varphi(s)y^*\}$ has an extension to a linear functional in $\bigcap_{s\in S}\{x^*\in X^*:~ s\cdot x^* =\varphi(s)x^*\}$;
  \item[(iv)] there is a bounded projection from $X^*$ onto $\bigcap_{s\in S}\{x^*\in X^*:~ s\cdot x^*=\varphi(s)x^*\}$ which commutes with any bounded linear operator from $X^*$ into $X^*$ commuting with the action of $S$ on $X$.
\end{itemize}
\end{thm}
\begin{proof}
(i)$\Longrightarrow$(ii) Consider  $RUC(S)^*$ with $w^*$-topology and suppose that $\mu$ is a $P(\varphi)$-invariant mean on a $P(\varphi)$-invariant subspace $X$ of $RUC(S)^*$. Let $Y$ be a subspace of  $RUC(S)^*$, contains all means on  $RUC(S)^*$ that they are extensions of $\mu$. The Hahn-Banach Theorem follows that $Y$ is non-empty and $w^*$-compactness and convexity of the set of all means on  $RUC(S)^*$ implies that it is  $w^*$-compact and convex.

Define the continuous affine action $T:P(\varphi)\times RUC(S)^{**}\longrightarrow RUC(S)^{**}$ by $T_s\nu=\ell_s^*\nu$, for all $\nu\in RUC(S)^{**}$ and $s\in P(\varphi)$, where $\ell_s^*$ is the adjoint of $\ell_s$. Continuity of $\ell_s$ implies that the maps $s\longmapsto T_s\nu$ and $\mu\longmapsto T_s\nu$  are continuous, for all $\nu\in  RUC(S)^*$ and $s\in P(\varphi)$.

For all $\widetilde{\mu}\in Y$ and $s\in P(\varphi)$,
$$\ell_s^*\widetilde{\mu}(I)=\widetilde{\mu}(\ell_sI)={\mu}(I)=1.$$

Hence, $\widetilde{\mu}$ is a mean on $RUC(S)^*$. Moreover,
\begin{eqnarray*}
  \ell_s^*\widetilde{\mu}(m) &=&\widetilde{\mu}\ell_s(m)= {\mu}\ell_s(m) \\
   &=& \mu(m),
\end{eqnarray*}
for all $m\in X$ and $s\in P(\varphi)$. Thus, $T_s(Y)\subseteq Y$, for every $s\in P(\varphi)$. Furthermore,
\begin{eqnarray*}
  T_sT_{s'}\nu(m) &=& T_s\ell_{s'}^*\nu(m) =\ell_s^*\ell_{s'}^*\nu(m)=\ell_{s'}^*\nu(\ell_sm)\\
   &=& \nu(\ell_{s'}\ell_sm)=\nu(\ell_{ss'}m)= \ell_{ss'}^*\nu(m)\\
   &=&  T_{ss'}\nu(m),
\end{eqnarray*}
for all $\nu\in RUC(S)^{**}$, $m\in RUC(S)^*$ and $s,s'\in P(\varphi)$. This means that $T_sT_{s'}=T_{ss'}$, for all $s,s'\in P(\varphi)$, i.e., $T_s$ is a representation on $RUC(S)^{**}$. Since $S$ is left $\varphi$-amenable, Corollary \ref{fpt} together with Remark \ref{rem1} implies that there exists $\widetilde{\mu}\in Y$ such that $T_s\widetilde{\mu}=\widetilde{\mu}$. Note that if in the definition of $T_y$ in the proof of Corollary \ref{fpt}, we replace $P(\varphi)$ by $S$, the relation \eqref{fpte2} becomes $h(sy')=\varphi(s)h(y')$, for every $s\in S$. Thus, $sy'=\varphi(s)y'$, for every $s\in S$.

(ii)$\Longrightarrow$(i) Again, consider $RUC(S)^*$ with the $w^*$-topology and suppose that  the left action $\ell:S\times RUC(S)^*\longrightarrow RUC(S)^*$ is defined by $\ell_sm=\varphi(s)m$, for all $m\in RUC(S)^*$ and $s\in S$. Clearly, the map $m\longmapsto\ell_sm$ is continuous. Set
$$X=\{m\in RUC(S)^*:m(\varphi)=1\}.$$

Note that by Hahn-Banach Theorem the set $\{m\in RUC(S)^*:m(\varphi)=1\}$ is non-empty. Therefore, there exists at least one $\textbf{m}\in X$ such that $\textbf{m}(\varphi)=1$. Set $Y=\mathbb{C}\textbf{m}$. Clearly, $Y$ is $P(\varphi)$-invariant. Define $\mu$ on $Y$ by $\mu(\lambda \textbf{m})=\lambda$, for every  $\lambda\in \mathbb{C}$. Then
$$\mu\ell_s(\textbf{m})=\mu(\varphi(s)\textbf{m})=\mu(\textbf{m}),$$
for all $s\in P(\varphi)$. This means that $\mu$ is $P(\varphi)$-invariant under $\ell$ and (ii) implies that any extension $\widetilde{\mu}$ is $S_\varphi$-invariant under $\ell$ i.e., $\widetilde{\mu}\ell_s=\varphi(s)\widetilde{\mu}$, for every $s\in S$. Then
$$\widetilde{\mu}\ell_s(\textbf{m})=\widetilde{\mu}(s\cdot \textbf{m})=\widetilde{\mu}(\varphi(s)\textbf{m}),$$
for every $s\in S$. Hence, $s\cdot \textbf{m}=\varphi(s)\textbf{m}$, for every $s\in S$, because  $\widetilde{\mu}$ separates points of $RUC(S)^*$. This implies that
$$\textbf{m}(f\cdot s)=s\cdot \textbf{m}(f)=\varphi(s)\textbf{m}(f),$$
for every $f\in RUC(S)$ and $s\in S$. Thus, $S$ is left $\varphi$-amenable.

(i)$\Longrightarrow$ (iii) Assume that $Y$ is a closed Banach submodule of $X\in\ _\varphi\mathcal{M}^S$, then the quotient Banach space $X/Y$ is a Banach $S$-bimodule.
 Set $$\mathcal{K}=\bigcap_{s\in S}\{y^*\in Y^*:~s\cdot y^*=\varphi(s)y^*\}\ \text{and}\ \mathcal{X}=\bigcap_{s\in S}\{x^*\in X^*:~ s\cdot x^* =\varphi(s)x^*\}.$$

 Suppose that $\theta\in\mathcal{K}$ and  $\widetilde{\theta}\in X^*$ is an extension of $\theta$. Then
   $\mathbf{\rho}:Y^\bot\longrightarrow (X/Y)^*$
   is an onto isometry and $S$-module morphism, where $$Y^\bot=\{x^*\in X^*|~ \langle y, x^*\rangle=0,~ \emph{\emph{for}}~ \emph{\emph{every}}~ y\in Y\}.$$

    Then
\begin{equation}
 (s\cdot\widetilde{\theta})(y)-(\widetilde{\theta}\cdot s)(y)  =  (s\cdot\widetilde{\theta}-\varphi(s)\widetilde{\theta})(y)=0,
\end{equation}
for all $s\in S$ and $y\in Y$.
Therefore $s\cdot\widetilde{\theta}-\varphi(s)\widetilde{\theta}\in Y^\bot$. Define $D:S\longrightarrow (X/Y)^*$ by $D(s)=\mathbf{\rho}(s\cdot\widetilde{\theta}-\varphi(s)\widetilde{\theta})$, for every $s\in S$. Then
$$D(st)=\rho(st\cdot\widetilde{\theta}-\varphi(st)\widetilde{\theta}),$$
for all $s,t\in S$. On the other hand,
\begin{eqnarray*}
 s\cdot D(t)+D(s)\cdot t  &=& s\cdot\rho(t\cdot\widetilde{\theta}-\varphi(t)\widetilde{\theta})+\rho(s\cdot\widetilde{\theta}-\varphi(s)\widetilde{\theta})\cdot t \\
   &=& \rho(st\cdot\widetilde{\theta}-\varphi(s)\varphi(t)\widetilde{\theta}) \\
   &=&D(st),
\end{eqnarray*}
for all $s,t\in S$. This implies that  $D$ is a derivation. (i) follows that there exists $\mathfrak{y}\in (X/Y)^*$ such that $D(s)=s\cdot \mathfrak{y}-\mathfrak{y}\cdot s$, for every $s\in S$. Surjectivity of $\rho$ leads to there exists  $\mathfrak{z}\in Y^\bot$  such that $D(s)=s\cdot\rho(\mathfrak{z})-\varphi(s)\rho(\mathfrak{z})$, for all $s\in S$. Now, set $\mathfrak{x}=\widetilde{\theta}-\mathfrak{z}$. Then \begin{equation}\label{}
    \nonumber
   (s\cdot (\widetilde{\theta}-\mathfrak{z})-\varphi(s)(\widetilde{\theta}-\mathfrak{z}))(y)=0,
\end{equation}
for all $s\in S$ and $y\in Y$.
Thus $\mathfrak{x}\in\mathcal{X}$. This completes the proof.

(iii)$\Longrightarrow$ (iv) Consider   $X^*\widehat{\otimes}X$ as a Banach $S$-bimodule by the following actions
\begin{equation}\label{2}
    (f\otimes x)\cdot s=f\otimes x\cdot s\hspace{1cm}\emph{\emph{and}}\hspace{1cm}s\cdot(f\otimes x)=f\otimes \varphi(s)x=\varphi(s)f\otimes x,
\end{equation}
for all $s\in S$, $x\in X$ and  $f\in X^*$. Set $\mathfrak{B}(X)=\{T\in\textbf{B}(X):s\bullet T=\varphi(s)T\}$ where ``$\bullet$" is the module product of $S$ on $\textbf{B}(X)$. Consider  the following sets
$$H:=\overline{\emph{\emph{lin}}}\Big{\{}T^*(f)\otimes x-f\otimes T(x)~:~T\in
{\mathfrak{B}}(X), f\in X^*, x\in X\Big{\}}$$
 and
$$K:=\overline{\emph{\emph{lin}}}\Big{\{}f\otimes x~:~f\in \mathcal{X}\Big{\}},$$
 where by $\overline{\emph{\emph{lin}}}$, we mean the closed linear span. Let $Y$ be the closed linear span of $H$ and $K$. Clearly, $H$ and $K$  are Banach $S$-submodules of $X^*\widehat{\otimes} X$. Therefore the quotient space $Y/H$ is a Banach $S$-submodule of $(X^*\widehat{\otimes}X)/H$. Let $\theta\in(X^*\widehat{\otimes}X)^*$ that satisfies  $\theta(f\otimes x)=f(x)$, for all $x\in X$ and  $f\in X^*$. Then
\begin{equation}\label{3}
       \theta(T^*(f)\otimes x-f\otimes T(x))= T^*(f)(x)-f(T(x))=0,
\end{equation}
for all $x\in X$ and all $f\in X^*$. This means that  $\theta\in H^\bot$. Pick $\vartheta\in((X^*\widehat{\otimes}X)/H)^*$ such that $\vartheta(y+H)=\theta(y)$, for every $y\in X^*\widehat{\otimes}X$. By   \eqref{2},  we have
\begin{eqnarray}\label{4}
\nonumber
  (s\cdot\vartheta-\vartheta\cdot s)(f\otimes x+H) &=&  \vartheta((f\otimes x+H) \cdot s)-\vartheta(s\cdot (f\otimes x+H))  \\
   \nonumber
 &=&  \vartheta(f\otimes x\cdot s+H)-\varphi(s)\vartheta(f\otimes x+H)\\
     &=&  \theta(f\otimes x\cdot s)-\varphi(s)\theta(f\otimes x),
\end{eqnarray}
for all $s\in S, x\in X$ and all $f\in X^*$.
Since $\alpha\in H^\bot$, if we apply \eqref{4}, then  $$s\cdot \vartheta=\varphi(s)\vartheta,$$
 for all $s\in S$ and all $\vartheta\in(Y/H)^*$. The part (ii) implies there is an extension $\widetilde{\vartheta}$ of $\vartheta$ such that $\widetilde{\vartheta}\in(X/H)^*$ and $s\cdot\widetilde{\vartheta}=\varphi(s)\widetilde{\alpha}$. Define $P(f)(x)= \widetilde{\vartheta}(f\otimes x+H)$, for all $x\in X$ and  $f\in X^*$. Clearly, $P$ is bounded. Moreover,
\begin{eqnarray}
 \nonumber
  (P\circ P(f))(x) &=&P(P(f))(x)= P(\widetilde{\vartheta}(f\otimes x+H))=P(f(x))\\
    \nonumber
   &=&  P(f)(x)
\end{eqnarray}
for all $x\in X$ and  $f\in X^*$. This shows that $P$ is a bounded projection from $X^*$ onto $X^*$. Let $T\in\mathfrak{B}(X)$, then by \eqref{3}, we have
\begin{eqnarray}
 \nonumber
 P T^*(f)(x) &=&  \widetilde{\vartheta}(T^*(f)\otimes x)=T^*(f)(x) \\
    \nonumber
 &=&   \vartheta(f\otimes T(x))=\widetilde{\vartheta}(f\otimes T(x)+H)= P(f)(T(x))\\
    \nonumber
    &=& T^*P(f)(x),
\end{eqnarray}
for all $x\in X$ and all $f\in X^*$.
Thus, $P$ commutes with every $T^*\in\mathfrak{B}(X^*)$.

(iv)$\Longrightarrow$(i) Assume that  $X=RUC(S)^*\widehat{\otimes}RUC(S)^*$, which $X$ becomes a Banach $S$-bimodule by the following actions
\begin{equation}\label{5}
    s\cdot(f\otimes g)=\varphi(s) f\otimes g\hspace{1cm}\emph{\emph{and}}\hspace{1cm}(f\otimes g)\cdot s=f\otimes g,
\end{equation}
for all $f,g\in RUC(S)^*$ and $s\in S$. Let $\mathcal{T}=\{r_s:~s\in S\}\cup \{l_s:~s\in S\}$ be a family of bounded linear operators from $X$ into $X$, such that
\begin{equation}\label{5}
    l_s(f\otimes g)=\varphi(s) f\otimes g\hspace{1cm}\emph{\emph{and}}\hspace{1cm}r_s(f\otimes g)=f\otimes g,
\end{equation}
for all $f,g\in RUC(S)^*$ and $s\in S$. Then
\begin{equation}\label{6}
    t\cdot l_s(f\otimes g)=(l_s\cdot t)(f\otimes g),
\end{equation}
and
\begin{equation}\label{6'}
 t\cdot r_s(f\otimes g)=(r_s\cdot t)(f\otimes g),
\end{equation}
for all $f,g\in RUC(S)^*$ and $s,t\in S$. Thus,  every member of $\mathcal{T}$ commutes with the action of $S$ on $X$. Then (iii) implies that there is a bounded surjective projection $P:X^*\longrightarrow X^*$ such that $PT^*=T^*P$, for all $T\in\mathcal{T}$.

 Define $\tau:X^*\longrightarrow X^*$ by $\tau(F)(f\otimes g)=F(g\otimes f)$, for all  $F\in X^*$ and $f,g\in RUC(S)^*$. Then
\begin{eqnarray}\label{6}
 \nonumber
 \tau(F\cdot s)(f\otimes g) &=& (F\cdot s)(g\otimes f)=\varphi(s)F(g\otimes f) \\
 \nonumber
     &=&\tau(F)(\varphi(s)f\otimes g) = \tau(F)(l_s(f\otimes g)) \\
     &=& l_s^*\tau(F)(f\otimes g)
\end{eqnarray}
for all $F\in X^*$, $f,g\in RUC(S)^*$ and $s,t\in S$. Consider the projective mapping $\pi:RUC(S)^*\widehat{\otimes}RUC(S)^*\longrightarrow RUC(S)^*$ defined by $\pi(f\otimes g)=fg$, for all  $f,g\in RUC(S)^*$
  Set $M=\tau^*(P^*(I\otimes I))$, where $I$ is the identity function in $RUC(S)^*$.  Since $\pi^{**}M\in RUC(S)^{***}$,
  \begin{equation}\label{fe1}
    s\cdot \pi^{**}M=\varphi(s)\pi^{**}M,
  \end{equation}
for every $s\in S$. Then by properties of $P$,  we have
\begin{eqnarray}\label{fe}
 \nonumber
 \pi^{**}M(F)&=& \pi^{**}\tau^*(P^*(I\otimes I))(F) =(P^*(I\otimes I))(\tau\pi^{*}(F))\\
 \nonumber
   &=&  (P\tau\pi^{*}F)(I\otimes I)=(\tau\pi^{*}(F))(I\otimes I)\\
   \nonumber
   &=& \pi^{*}(F)(I\otimes I)= F\pi^{**}(I\otimes I)\\
   &=&F(I),
\end{eqnarray}
for every $F\in RUC(S)^{**}$. Now, we set $m=\pi^{**}M|_{RUC(S)}$.
Then \eqref{fe1} implies that $s\cdot m=\varphi(s)m$, for every $s\in S$. Thus,
$$s\cdot m(f)=m(f\cdot s)=\varphi(s)m(f),$$
for every $f\in RUC(S)$ and $s\in S$. The relation \eqref{fe} shows that, for any $f\in RUC(S)$, $\pi^{**}M(f)=I(f)=1$. This implies that $m(\varphi)=1$ and consequently, $S$ is left $\varphi$-amenable.
\end{proof}

For the right case, similarly, by defining the right Hahn-Banach Property, we can prove the above theorem.

%-----------------------------------------------------------------------------------------------------------------------------------------------------
\section*{Acknowledgment}
The authors would like to thank the referee for careful reading of the paper
and for his/her useful suggestions, which greatly improved the presentation of the
paper.
%-----------------------------------------------------------------------------------------------------------------------------------------------------

%-----------------------------------------------------------------------------

\end{document}